\documentclass[a4paper,leqno,final]{article}
\usepackage{etex}

\usepackage[british]{babel}

\usepackage[T1]{fontenc}
\usepackage{ae,aecompl}
\usepackage{amsmath}
\usepackage{amsthm}
\usepackage{amssymb}
\usepackage[all]{xy}
\usepackage{rotating}
\usepackage{enumitem}
\usepackage{aliascnt}
\usepackage{hyperref}
\usepackage{mathtools}
\usepackage{subfig}
\usepackage{graphicx}
\usepackage{longtable}
\usepackage{tabularx}
\usepackage{booktabs}
\usepackage{microtype}
\usepackage{ifpdf}

\usepackage{algorithm}
\usepackage{algorithmic}

\newcommand{\makedate}{\today}

\usepackage{tikzexternal}
\tikzsetfigurename{urbstats}

\newtheorem{theorem}{Theorem}[section]

\newaliascnt{lemma}{theorem}
\newtheorem{lemma}[lemma]{Lemma}
\aliascntresetthe{lemma}

\newaliascnt{proposition}{theorem}
\newtheorem{proposition}[proposition]{Proposition}
\aliascntresetthe{proposition}

\newaliascnt{corollary}{theorem}
\newtheorem{corollary}[corollary]{Corollary}
\aliascntresetthe{corollary}

\newaliascnt{remark}{theorem}
\newtheorem{remark}[remark]{Remark}
\aliascntresetthe{remark}

\newaliascnt{definition}{theorem}
\newtheorem{definition}[definition]{Definition}
\aliascntresetthe{definition}

\newaliascnt{example}{theorem}
\newtheorem{example}[example]{Example}
\aliascntresetthe{example}

\newaliascnt{conjecture}{theorem}
\newtheorem{conjecture}[conjecture]{Conjecture}
\aliascntresetthe{conjecture}

\newaliascnt{openproblem}{theorem}

\aliascntresetthe{openproblem}

\newaliascnt{exercise}{theorem}

\aliascntresetthe{exercise}

\newaliascnt{algo}{theorem}

\aliascntresetthe{algo}

\newaliascnt{assumption}{theorem}

\aliascntresetthe{assumption}

\newaliascnt{notation}{theorem}

\aliascntresetthe{notation}

\addto\extrasbritish{

}
\newcounter{namedenum}
  {\end{list}}
\newcommand{\Magma}{\textsc{Magma}}
\newcommand{\NN}{\mathbb{N}}
\newcommand{\ZZ}{\mathbb{Z}}
\newcommand{\Atoms}{\mathcal{A}}
\newcommand{\Simples}{\mathcal{D}}
\newcommand{\pSimples}{\Simples^{\!{}^{\circ}\!}}

\newcommand{\s}{\sigma}

\DeclareMathOperator{\cl}{cl}
\newcommand{\pd}{\mathrm{pd}}
\newcommand{\PSeq}{\textsc{PSeq}}

\newcommand{\from}{\colon\!}
\newcommand{\URB}{\mathsf{URB}}
\newcommand{\Word}{\mathsf{Word}}
\newcommand{\id}{\mathbf{1}}
\newcommand{\Artin}[1]{\mathsf{#1}}
\newcommand{\rev}{\mathrm{rev}}
\newcommand\pto{\mathrel{\ooalign{\hfil$\mapstochar\mkern5mu$\hfil\cr$\to$\cr}}}
\let\leq\leqslant
\let\le\leqslant
\let\geq\geqslant
\let\ge\geqslant

\newcolumntype{L}{>{\(}l<{\)}}
\newcolumntype{C}{>{\(}c<{\)}}
\newcolumntype{R}{>{\(}r<{\)}}

\graphicspath{{images/}}

\begin{document}

\title{Normal forms of random braids}
\author{Volker Gebhardt\footnotemark[1]{ }\textsuperscript{,}\footnotemark[2]\hspace{0.3em} and Stephen Tawn\footnotemark[1]}
\date{\makedate}

\renewcommand{\thefootnote}{\fnsymbol{footnote}}
\footnotetext[1]{Both authors acknowledge support under Australian Research Council's Discovery Projects funding scheme (project number DP1094072).}
\footnotetext[2]{Volker Gebhardt acknowledges support under the Spanish Project MTM2010-19355.}

\maketitle
\renewcommand{\thefootnote}{\arabic{footnote}}

\begin{abstract}
\noindent
Analysing statistical properties of the normal forms of random braids, we observe that, except for an initial and a final region whose lengths are uniformly bounded (that is, the bound is independent of the length of the braid), the distributions of the factors of the normal form of sufficiently long random braids depend neither on the position in the normal form nor on the lengths of the random braids.
Moreover, when multiplying a braid on the right, the expected number of factors in its normal form that are modified, called the \emph{expected penetration distance}, is uniformly bounded.

We explain these observations by analysing the growth rates of two regular languages associated to normal forms of elements of Garside groups, respectively to the modification of a normal form by right multiplication.

A universal bound on the expected penetration distance in a Garside group yields in particular an algorithm for computing normal forms that has linear expected running time.
\end{abstract}


\section{Introduction}\label{S:Intro}%

Explicit computations play an increasingly important role in most areas of algebra; the study of braids is no exception.
In many situations, computations with braids involve choosing braids at random:  Some algorithms explicitly require a random braid to be generated; this is the case, for instance, in cryptographic protocols based on the braid group \cite{AAG,KoEtAl2000}.  At other times, a large collection of typical examples is to be generated; this is usually the case in computational experiments supporting theoretical research.

As $B_n$, the group of braids on~$n$ strands, is infinite, choosing braids at random is not a trivial task. There are various natural ways of choosing elements of~$B_n$ at random, and different approaches will yield different probability distributions on~$B_n$.  For both, computational experiments and applications (especially applications in cryptography), it is important to understand the statistical probabilities of samples of random elements generated using a particular method.%
\medskip

We consider in the following the \emph{braid monoid} $B_n^+$ defined by the presentation
\begin{equation}\label{E:BraidMonoid}
   B_n^+ =  \left\langle \s_1,\s_2,\ldots,\s_{n-1}
             \;\Bigg| \begin{array}{r@{\ =\ }lr}
                          \s_i\s_j & \s_j\s_i             & (1\le i<j<n) \\[0.5ex]
                  \s_i\s_{i+1}\s_i & \s_{i+1}\s_i\s_{i+1} & (1\le i<n-1)
               \end{array}
        \right\rangle^+
   \;,
\end{equation}
which follows Artin's presentaion for the braid group \cite{artin_braid_groups}.
As the relations of $B_n^+$ are homogeneous, the number of generators occurring in any expression of $x\in B_n^+$ is well-defined; we call this number the \emph{length} $|x|$ of~$x$.  We can then fix a non-negative integer $k$ and generate an element $x\in B_n^+$ of length~$k$.
More specifically, there are two possibilities:
\begin{itemize}
\item[(A)] For $i=1,2,\ldots,k$ independently choose $a_i\in\Atoms = \{\s_1,\s_2,\ldots,\s_{n-1}\}$ with a uniform
  probability distribution on $\Atoms$, or equivalently, consider the set $\Atoms^*|_k$ of all \emph{words} of length
  $k$ over the alphabet $\Atoms$, and choose an element of $\Atoms^*|_k$ at random with a uniform probability
  distribution on this set.
\item[(B)] Consider the set $B_n^+|_k = \{ x \in B_n^+ : |x| = k\}$ and choose an element of $B_n^+|_k$ at
  random with a uniform probability distribution on this set.
\end{itemize}
We will refer to (A) as \emph{generating uniformly random words}, and to (B) as \emph{generating uniformly random braids}.  We will write $\Word_k$, respectively $\URB_k$, for the corresponding probability measures on $B_n^+$.
Since the number of different words in $\Atoms^*|_k$ that represent the same element $x$ of $B_n^+|_k$ depends on $x$, generating uniformly random words results in a distribution of \emph{braids} which is very far from being uniform on $B_n^+|_k$.
Generating uniformly random braids is not easy; an algorithm whose time- and space-complexities are polynomial in both $n$ and~$k$ was given in~\cite{URB}.
\medskip

In this paper we analyse the generation of uniformly random braids and the generation of uniformly random words regarding some properties of the generated samples of braids.
The \emph{Garside normal form} defines a canonical way of expressing a braid as a sequence of permutations, so a probability distribution on the braid group induces a sequence of probability distributions on the symmetric group.  We are particularly interested in how the resulting distributions on the symmetric group depend on the position in this sequence.

The structure of the paper is as follows:
\autoref{S:Background} recalls
the Garside normal form; experts may skip this section.
\autoref{S:NF} contains our analysis of the normal forms of random braids.
In \autoref{SS:NF:StableRegion} we observe that there is a ``stabilisation'' occurring in the normal forms of long random braids in the sense that for sufficiently long braids the distributions on the symmetric group induced by the factors of the normal form depend neither on the position in the normal form nor on the lengths of the random braids, except for an initial and a final region whose lengths are uniformly bounded.
In \autoref{SS:NF:BoundedPD}, we give an explanation for this stabilisation phenomenon by demonstrating that the expected number of factors of the normal form of a braid that are modified when multiplying the braid on the right is uniformly bounded.
Finally, in \autoref{S:GarsideGroups}, we extend our analysis to general Garside groups and establish a criterion for deciding whether phenomena similar to the ones described above occur in a given Garside group.

\section{Background}\label{S:Background}%

This section contains a brief summary of the main notions referred to in the paper.  Specifically, we will recall Garside monoids and the Garside normal form.
For details and proofs we refer to~\cite{braid_epsteinetal,Dehornoy02}.

In a cancellative monoid $M$ with unit $\id$, we can define the \emph{prefix} partial order:
For $x,y\in M$, we say $x\preccurlyeq y$ if there exists $c\in M$ such that $xc=y$.
Similarly, we define the \emph{suffix} partial order by saying that $x\succcurlyeq y$ if there exists $c\in M$ such that $x=cy$.
We call $s\in M$ an \emph{atom}, if $s=ab$ (with $a,b\in M$) implies $a=\id$ or $b=\id$.
We write $\Atoms$ for the set of atoms of $M$.

A cancellative monoid $M$ is called a \emph{Garside monoid of spherical type}, if it is a lattice (that is, least common multiples and greatest common divisors exist and are unique) with respect to $\preccurlyeq$ and with respect to $\succcurlyeq$, if there are no strict infinite descending chains with respect to either $\preccurlyeq$ or $\succcurlyeq$, and if there exists an element $\Delta\in M$, such that
$\Simples=\{s\in M \mid s\preccurlyeq\Delta\}=\{s\in M \mid \Delta\succcurlyeq s\}$ is finite and generates $M$.
In this case, we call $\Delta$ a \emph{Garside element}, the elements of $\Simples$ the \emph{simple elements} (with respect to $\Delta$)
and the elements of $\pSimples = \Simples\setminus\{\id,\Delta\}$ the \emph{proper simple elements} (with respect to $\Delta$).
Moreover, we denote the $\preccurlyeq$-gcd and the $\preccurlyeq$-lcm of $x,y\in M$ by $x\wedge y$ respectively $x\vee y$.
It follows that, for $s\in \Simples$, there exists a unique element $\partial s\in \Simples$ such that $s\,\partial s=\Delta$.
\medskip

We assume for the rest of this section that $M$ is a Garside monoid of spherical type.
Since $\Simples$ generates $M$, every element $x\in M$ can be written in the form $x=x_1x_2\cdots x_m$ with $x_1,x_2,\ldots,x_m\in \Simples$.
The representation as a product of this form can be made unique by requiring that each simple factor is non-trivial and maximal with respect to $\preccurlyeq$.  More precisely, we say that $x=x_1x_2\cdots x_m$ is in \emph{(left) (Garside) normal form}, if $x_m\ne \id$ and if $x_i = \Delta\wedge(x_ix_{i+1}\cdots x_m)$ for $i=1,2,\ldots,m$.
Equivalently, we can require
\begin{equation}\label{E_NormalForm}
 x_m\ne \id \quad\text{and}\quad \partial x_i \wedge x_{i+1} = \id \text{\, for \,} i=1,2,\ldots,m-1 \;.
\end{equation}

If a word is in normal form then all occurrences of $\Delta$ must be at the
start, hence the normal form of $x$ is of the form $\Delta^k x_1 x_2 \cdots x_l$
where $x_i \in \pSimples$.  We say that
$\inf(x) = k$ is the \emph{infimum} of $x$, $\cl(x) = l$ is the \emph{canonical
length} of $x$, and $\sup(x) = k + l$ is the \emph{supremum} of $x$.

As $M$ satisfies the Ore conditions, it embeds into its quotient group $Q(M)$.
Conjugation by $\Delta$ gives a bijection $\tau \from \Simples \to
\Simples$ and, as $\Simples$ is finite, this implies that some power of $\Delta$ is central.
Hence, for every element $x$ of $Q(M)$ there exists an integer $k$ such that
$\Delta^k x$ lies in $M$, and so we can extend the
normal form to the quotient group $Q(M)$.

Of particular interest to us in \autoref{S:NF} will be the maps
projecting onto the $i$-th non-$\Delta$ factor from the left, respectively from the right, of the
normal form.
If $\Delta^k x_1 x_2 \cdots x_l$ is in normal form, we define
\begin{align*}
  \lambda_i(x) &= \begin{cases}
    x_i     & \text{ for $i = 1, 2, \ldots, l$} \\
    \id     & \text{ otherwise}
  \end{cases} & \text{and }\,
  \rho_i(x) &= \begin{cases}
    x_{l+1-i} & \text{ for $i = 1, 2, \ldots, l$} \\
    \id       & \text{ otherwise}
  \end{cases} \;\;.
\end{align*}

\subsubsection*{Classical Garside structure for the braid group}
The braid monoid $B_n^+$ defined by the presentation \eqref{E:BraidMonoid} is a Garside monoid of spherical type whose quotient group is the braid group $B_n$ on $n$ strands.
It was in the context of the braid monoid that the left normal form was first used by Garside\cite{Garside} to solve the word and conjugacy problems in the braid group.
The monoid $B_n^+$ is also referred to as the \emph{classical Garside monoid} for $B_n$.
The atoms of $B_n^+$ are the generators $\s_1,\s_2,\ldots,\s_{n-1}$, and the Garside element $\Delta$ of $B_n^+$ is the so-called \emph{half-twist}, the positive braid in which any two strands cross exactly once.
The simple braids are exactly those positive braids in which any two strands cross at most once.  In particular, a simple braid is characterised by the permutation which it induces on the strands, whence the set $\Simples$ is in bijection to the symmetric group $S_n$.

Given $x\in B_n^+$, we define the \emph{starting set} of $x$ as $S(x)=\{a\in \Atoms : a\preccurlyeq x\}$ and the \emph{finishing set} of $x$ as $F(x)=\{a\in \Atoms : x\succcurlyeq a\}$.

We remark that, as an element $x\in B_n^+$ is simple if and only if it is square-free, that is, if and only if it cannot be written as $x=ua^2v$ with $u,v\in B_n^+$ and $a\in\Atoms$, the conditions characterising normal forms can be expressed in terms of starting and finishing sets:
For any $x\in \Simples$, we have $F(x)\cap S(\partial x)=\emptyset$ and $F(x)\cup S(\partial x)=\Atoms$~\cite[Lemma 4.2]{Charney1993}.  For $u,v\in \Simples$, one therefore has $\partial u\wedge v=\id$ if and only if $S(v)\subseteq \Atoms\setminus S(\partial u)= F(u)$.

\section{Normal form}\label{S:NF}%

In this section we will investigate the normal form of random
elements.  For our experiments, we constructed and analysed samples
of 9999 elements of $B_n^+$ for each combination of
number of strands
$n\in\{ 5, 10, 15, 20, 25, 30\}$
and word length
$k\in\{4, 8, 12, 16, 24, 32, 48, 64, 96, 128, 192, 256, 512, 1024, 2048\}$
for both uniformly random words and uniformly random braids.
For uniformly random words we also analysed samples with a word length of 4096.
The samples of uniformly random braids were constructed using an implementation of the algorithm described in~\cite{URB} by the first author; the rest of the computations were done using a development version of \Magma~\cite{magma} V2.19.

Using these samples we will investigate the distribution of simple
factors along the normal form of the elements, that is, we will look at
how the induced probability measures $\lambda_{i*}(\Word_k)$,
$\lambda_{i*}(\URB_k)$, $\rho_{i*}(\Word_k)$ and $\rho_{i*}(\URB_k)$
on the set of simple elements vary with $i$.  This will lead us to
investigate how the normal form changes when an element is multiplied by
an atom.

\subsection{Stable region}\label{SS:NF:StableRegion}%

The fact that there are a large number of simple elements makes it
impractical to look directly at the distribution at each position of
the normal form.  So, we will use several invariants instead, namely the word
length and the starting and finishing sets, to indirectly probe these
distributions.

\subsubsection*{Word length}

\begin{figure}[!p]
  \centering
  \subfloat[Uniformly random words] {
    \begin{tikzpicture}[spy using outlines={
          circle,
          magnification=7,
          connect spies
      }]
      \begin{axis}[
          height=0.5\textwidth,
          width=0.96\textwidth,
          xlabel=Factor,
          ylabel=Mean factor length,
          cycle multi list={
%
%
            black,brown,red,blue\nextlist
            solid\nextlist
            mark=none
          },
          legend reversed=true,
        ]
        
        \addplot table [x=pos, y=LEN] {tables/word-len-n10-k2048.tab}
            node[pos=0.6, anchor=south west, black] {$n=10$};
        \addplot table [x=pos, y=LEN] {tables/word-len-n10-k1024.tab};
        \addplot table [x=pos, y=LEN] {tables/word-len-n10-k512.tab};
        \addplot table [x=pos, y=LEN] {tables/word-len-n10-k256.tab};
        
        \addplot table [x=pos, y=LEN] {tables/word-len-n30-k2048.tab}
            node[pos=0.5, anchor=south west, black] {$n=30$};
        \addplot table [x=pos, y=LEN] {tables/word-len-n30-k1024.tab};
        \addplot table [x=pos, y=LEN] {tables/word-len-n30-k512.tab};
        \addplot table [x=pos, y=LEN] {tables/word-len-n30-k256.tab};
        
        \addlegendimage{empty legend}
        \addlegendentry{2048}
        \addlegendentry{1024}
        \addlegendentry{512}
        \addlegendentry{256}
        \addlegendentry{Word length}

        \coordinate (spypoint) at (axis cs:6.3,7.1); 
        \coordinate (magnifyglass) at (axis cs:180,20);
      \end{axis}

      \spy [gray, size=2.5cm] on (spypoint)
                              in node[fill=white] at (magnifyglass);
    \end{tikzpicture}
  }

  \subfloat[Uniformly random braids]{
    \begin{tikzpicture}[spy using outlines={
          circle,
          magnification=7,
          connect spies
      }]
      \begin{axis}[
          height=0.5\textwidth,
          width=0.96\textwidth,
          xlabel=Factor,
          ylabel=Mean factor length,
          cycle multi list={
            black,brown,red,blue\nextlist
            solid\nextlist
            mark=none
          },
          legend reversed=true,
        ]

        \addplot table [x=pos, y=LEN] {tables/URB-len-n10-k2048.tab}
            node[pos=0.7, anchor=south west, black] {$n=10$};
        \addplot table [x=pos, y=LEN] {tables/URB-len-n10-k1024.tab};
        \addplot table [x=pos, y=LEN] {tables/URB-len-n10-k512.tab};
        \addplot table [x=pos, y=LEN] {tables/URB-len-n10-k256.tab};
        
        \addplot table [x=pos, y=LEN] {tables/URB-len-n30-k2048.tab}
            node[pos=0.6, anchor=south west, black] {$n=30$};
        \addplot table [x=pos, y=LEN] {tables/URB-len-n30-k1024.tab};
        \addplot table [x=pos, y=LEN] {tables/URB-len-n30-k512.tab};
        \addplot table [x=pos, y=LEN] {tables/URB-len-n30-k256.tab};
        
        \addlegendimage{empty legend}
        \addlegendentry{2048}
        \addlegendentry{1024}
        \addlegendentry{512}
        \addlegendentry{256}
        \addlegendentry{Word length}
        
        \coordinate (spypoint) at (axis cs:6.4,3.25); 
        \coordinate (magnifyglass) at (axis cs:180,12);
      \end{axis}

      \spy [gray, size=2.5cm] on (spypoint) 
                              in node[fill=white] at (magnifyglass);
    \end{tikzpicture}
  }
  \caption{Mean factor length.}
  \label{Fig:len}
  \bigskip\bigskip
%
    \begin{tikzpicture}
      \begin{axis}[
          height=0.5\textwidth,
          width=0.95\textwidth,
          xlabel=$n$,
          ylabel=Mean factor length,
          xmin=0,
          ymin=0,
          xmax=30,
          ymax=20,
          legend pos= north west
        ]
        \addplot table [x=n, y=URB-PLEN] {tables/stable-length.tab};
        \addlegendentry{uniformly random braids}
        \addplot table [x=n, y=word-PLEN] {tables/stable-length.tab};
        \addlegendentry{uniformly random words}
      \end{axis}
    \end{tikzpicture}
  \caption{Mean factor length inside stable region.}
  \label{Fig:stable-length}
\end{figure}

\autoref{Fig:len} shows how the mean factor length varies along the
word.  We observe that, provided the word is long enough, the word
can be divided into three regions: An initial region where the word
length of the factors is rapidly decreasing; a stable region where the
word length is constant; and a terminal region where it drops to zero.
Moreover, the shape and size of this initial region is independent of
the word length.  The same structure occurs for the samples not shown
here.

The variation in the canonical length within each sample has
``smeared out'' the terminal region, causing it to grow in size as the
word length increases.  If we were to draw right justified plots,
that is, if the $x$-axis was the distance from the end, then you would
see that, like the initial regions, the terminal regions have a
constant size and shape for sufficiently long words.

\begin{figure}[!h]
  \centering
  \subfloat[Uniformly random words] {
    \centering
    \begin{tikzpicture}
      \begin{axis}[
          ycomb,
          height=0.5\textwidth,
          width=0.45\textwidth,
          no markers,
          xlabel=Factor length,
          ylabel=Relative frequency,
          yticklabel style={/pgf/number format/fixed},
          scaled ticks=false
        ]
        \addplot table [x=length, y=freq] {tables/Word-len-stable-n30.tab};
      \end{axis}
    \end{tikzpicture}
  }
  \hfill
  \subfloat[Uniformly random braids]{
    \centering
    \begin{tikzpicture}
      \begin{axis}[
          ycomb,
          height=0.5\textwidth,
          width=0.45\textwidth,
          no markers,
          xlabel=Factor length,
          ylabel=Relative frequency,
          tick label style={/pgf/number format/fixed}
        ]
        \addplot table [x=length, y=freq] {tables/URB-len-stable-n30.tab};
      \end{axis}
    \end{tikzpicture}
  }
  \caption{Distribution of factor lengths in the stable region for $n = 30$.}
  \label{Fig:lengths-stable}
\end{figure}
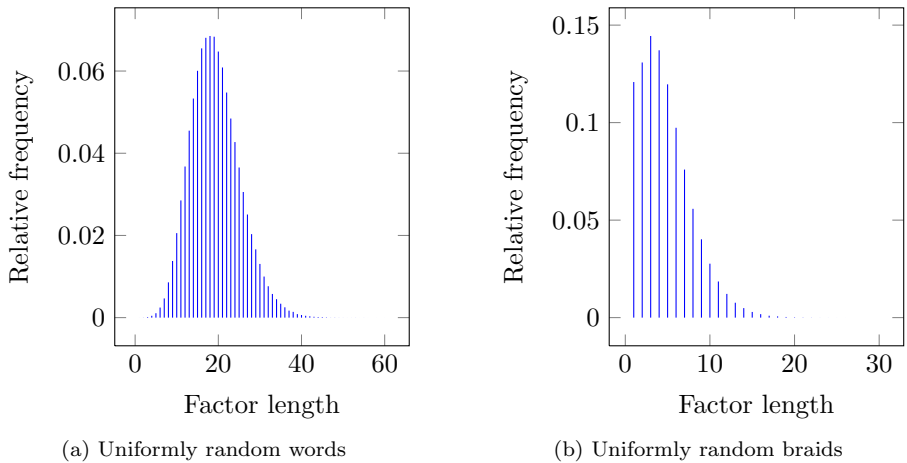

\autoref{Fig:stable-length} shows how the mean factor length inside the stable region depends on~$n$ for both uniformly random words and uniformly random braids, and \autoref{Fig:lengths-stable} shows the distributions of factor lengths in the stable region for $n=30$ for both uniformly random words and uniformly random braids.

For uniformly random words, the observed mean factor lengths are consistent with a linear function in~$n$ (the best fit of a model of the form~$n^c$ is obtained for $c\approx0.9952$), whereas for uniformly random braids the mean factor length grows much more slowly; the best fit of a model of the form~$n^c$ is obtained for $c\approx0.4494$.

The data indicates that normal forms of uniformly random words are much more ``densely packed'' (in the sense that each simple factor on average contains more crossings) than those of uniformly random braids, and that this difference becomes more pronounced with increasing~$n$.  This is consistent with the fact that the distribution of braids obtained by choosing uniformly random words is biassed towards multiples of the Garside elements of standard parabolic subgroups, that is, the lcms of subsets of $\Atoms$~\cite{URB}.  Such lcms have the maximal number of representing words, as they can be rewritten using \emph{all} braid relations between the generators involved.  On the other hand, these lcms also have the maximal possible word length of all simple elements in the standard parabolic subgroups, that is, they yield the densest possible packing of the involved generators into simple factors.

\subsubsection*{Starting and finishing sets}

\begin{figure}[!htbp]
  \centering
  \subfloat[Uniformly random words] {
    \centering
    \begin{tikzpicture}
      \begin{axis}[
          height=0.45\textwidth,
          width=0.45\textwidth,
          no markers,
          xlabel=Distance from start,
          ylabel=Relative frequency,
          xmin=0,
          xmax=50,
          xtick={0,25,50},
          ymin=0,
          ymax=0.75,
        ]
        \addplot table [x=pos, y=gen1] {tables/Word-start-L-n30-k2048.tab};
        \addplot table [x=pos, y=gen8] {tables/Word-start-L-n30-k2048.tab};
        \addplot table [x=pos, y=gen15] {tables/Word-start-L-n30-k2048.tab};
      \end{axis}
    \end{tikzpicture}
    \begin{tikzpicture}
      \begin{axis}[
          height=0.45\textwidth,
          width=0.45\textwidth,
          no markers,
          yticklabel pos=right,
          xlabel=Distance from end,
          xmin=0,
          xmax=50,
          xtick={0,25,50},
          x dir=reverse,
          ymin=0,
          ymax=0.75,
        ]
        \addlegendimage{empty legend}
        \addlegendentry{Generator}
        \addplot table [x=pos, y=gen1] {tables/Word-start-R-n30-k2048.tab};
        \addlegendentry{1}
        \addplot table [x=pos, y=gen8] {tables/Word-start-R-n30-k2048.tab};
        \addlegendentry{8}
        \addplot table [x=pos, y=gen15] {tables/Word-start-R-n30-k2048.tab};
        \addlegendentry{15}
      \end{axis}
    \end{tikzpicture}
  }

  \subfloat[Uniformly random braids]{
    \centering
    \begin{tikzpicture}
      \begin{axis}[
          height=0.45\textwidth,
          width=0.45\textwidth,
          no markers,
          xlabel=Distance from start,
          ylabel=Relative frequency,
          xmin=0,
          xmax=50,
          xtick={0,25,50},
          ymin=0,
          ymax=0.35,
        ]
        \addplot table [x=pos, y=gen1] {tables/URB-start-L-n30-k2048.tab};
        \addplot table [x=pos, y=gen8] {tables/URB-start-L-n30-k2048.tab};
        \addplot table [x=pos, y=gen15] {tables/URB-start-L-n30-k2048.tab};
      \end{axis}
    \end{tikzpicture}
    \begin{tikzpicture}
      \begin{axis}[
          height=0.45\textwidth,
          width=0.45\textwidth,
          no markers,
          yticklabel pos=right,
          xlabel=Distance from end,
          xmin=0,
          xmax=50,
          xtick={0,25,50},
          x dir=reverse,
          ymin=0,
          ymax=0.35,
        ]
        \addlegendimage{empty legend}
        \addlegendentry{Generator}
        \addplot table [x=pos, y=gen1] {tables/URB-start-R-n30-k2048.tab};
        \addlegendentry{1}
        \addplot table [x=pos, y=gen8] {tables/URB-start-R-n30-k2048.tab};
        \addlegendentry{8}
        \addplot table [x=pos, y=gen15] {tables/URB-start-R-n30-k2048.tab};
        \addlegendentry{15}
      \end{axis}
    \end{tikzpicture}
  }
  \caption{Relative frequency of a generator being in the starting
    set for $n = 30$ and word length $= 2048$.}
  \label{Fig:starting}
  \bigskip\bigskip
  \subfloat[Uniformly random words] {
    \centering
    \begin{tikzpicture}
      \begin{axis}[
         ycomb,
          height=0.5\textwidth,
          width=0.45\textwidth,
          no markers,
          xlabel=Generator,
          ylabel=Relative frequency,
          xmin=0,
          xmax=30,
          ymin=0,
          ymax=0.5,
        ]
        \addplot table [x=gen, y=fin] {tables/Word-start-stable-n30.tab};
        \addplot table [x=gen, y=start] {tables/Word-start-stable-n30.tab};
      \end{axis}
    \end{tikzpicture}
  }\hfill
  \subfloat[Uniformly random braids]{
    \centering
    \begin{tikzpicture}
      \begin{axis}[
         ycomb,
         legend image code/.code={
            \draw (0,0) -- (0.1,0);
          },
          height=0.5\textwidth,
          width=0.45\textwidth,
          no markers,
          yticklabel pos=left,
          xlabel=Generator,
          ylabel=Relative frequency,
          xmin=0,
          xmax=30,
          ymin=0,
          ymax=0.5,
        ]
        \addplot table [x=gen, y=fin] {tables/URB-start-stable-n30.tab};
        \addlegendentry{Finishing set}
        \addplot table [x=gen, y=start] {tables/URB-start-stable-n30.tab};
        \addlegendentry{Starting set}
      \end{axis}
    \end{tikzpicture}
  }
  \caption{Relative frequency of a generator being in the starting
    and finishing set in the stable region for $n = 30$.}
  \label{Fig:starting-stable}
\end{figure}

\autoref{Fig:starting} shows, for a given generator, the relative
frequency with which that generator lies in the starting set for each
canonical factor.  To avoid the problem with the variation in
canonical length smearing out the end of the words we have drawn a
left justified plot for the beginning and a right justified plot for
the end of the word.  The plots not shown here for different values of~$n$,
different word lengths, different generators and for the
finishing sets all have a similar shape.  As we saw for the mean
factor length, there is an initial region, a stable region and a
terminal region.  In \autoref{Fig:starting} for uniformly random
braids and generator~$\sigma_{15}$ there is a local minimum around the 10th
factor.  Nevertheless, the size and shape of the initial, and
terminal, regions remains fixed once the word length is sufficiently
long.  Furthermore, the sizes of these regions are consistent with the
sizes observed for the mean factor length.

One clear difference between uniformly random words and
uniformly random braids lies in the frequency with which individual generators occur in the starting and
finishing sets; see \autoref{Fig:starting-stable}.
For uniformly random words the frequency is mostly independent of the generator, except for generators ``at the edges of the braid'' (that is, generators $\sigma_i$ with $i$ close to $1$ or $n-1$) which occur more frequently.
For uniformly random braids, on the other hand, the frequency varies very strongly between generators, with generators at the edges ($\sigma_1$ and $\sigma_{n-1}$) occurring very rarely and the frequency continuously increasing towards the middle and generators $\sigma_i$ with $i\approx\frac{n}2$ occurring most frequently.

\subsubsection*{Combining mean word length and starting\,/\,finishing set frequencies}

Given a sample of random braids, we consider the mean factor length and the relative frequency of each generator being in the starting set, respectively in the finishing set, as functions $f$ of the position $p$ in the non-$\Delta$ factors of the normal form.

For each of these functions $f$, we identify the interval $[p_1,p_2]$ that minimises the ratio $\frac{|f([p_1,p_2])|}{|[p_1,p_2]|}$, where, for $S\subseteq\mathbb R$, we define $|S| = \max(S) - \min(S)$.
(Intuitively, this procedure locates the ``most horizontal part'' of the graph of~$f$.)
We then fit a linear model $\widetilde{f}$ to $f|_{[p_1,p_2]}$ and accept the interval $[p_1,p_2]$ as stable region for~$f$ if $|\widetilde{f}([p_1,p_2])| < 0.1\cdot |f([p_1,p_2])|$; otherwise we consider the stable region for~$f$ as empty.  (Intuitively, this procedure ensures that the trend in the restriction of~$f$ to the interval $[p_1,p_2]$ is small compared to the statistical fluctuations of~$f$ on the interval.)

The stable region of the sample is taken to be the intersection of the stable regions for all the functions.
\medskip

\autoref{Fig:stable-start-k} and \autoref{Fig:stable-start-n} show the start of the stable region, determined as described above, as a function of $n$ and the word length for both uniformly random words and uniformly random braids.
The data shows that, for fixed $n$ and a given method of generating random braids, stable regions exist for sufficiently long random braids, and that their starting positions do not depend on the word length of the random braids.

\begin{figure}[!htbp]
  \centering
  \subfloat[Uniformly random words] {
    \centering
    \begin{tikzpicture}
      \begin{axis}[
          height=0.5\textwidth,
          width=0.96\textwidth,
          xlabel=Word length,
          ylabel=Distance from start,
          xmin=0,
          ymin=0,
          xmax=4096
        ]
        \addplot table [x=k, y=5] {tables/Word-stable-start.tab};
        \addlegendentry{$n = 5$}
        \addplot table [x=k, y=10] {tables/Word-stable-start.tab};
        \addlegendentry{$n = 10$}
        \addplot table [x=k, y=15] {tables/Word-stable-start.tab};
        \addlegendentry{$n = 15$}
        \addplot table [x=k, y=20] {tables/Word-stable-start.tab};
        \addlegendentry{$n = 20$}
        \addplot table [x=k, y=25] {tables/Word-stable-start.tab};
        \addlegendentry{$n = 25$}
        \addplot table [x=k, y=30] {tables/Word-stable-start.tab};
        \addlegendentry{$n = 30$}
      \end{axis}
    \end{tikzpicture}
  }

  \subfloat[Uniformly random braids]{
    \centering
    \begin{tikzpicture}
      \begin{axis}[
          height=0.5\textwidth,
          width=0.96\textwidth,
          xlabel=Word length,
          ylabel=Distance from start,
          xmin=0,
          ymin=0,
          xmax=4096,
        ]
        \addplot table [x=k, y=5] {tables/URB-stable-start.tab};
        \addlegendentry{$n = 5$}
        \addplot table [x=k, y=10] {tables/URB-stable-start.tab};
        \addlegendentry{$n = 10$}
        \addplot table [x=k, y=15] {tables/URB-stable-start.tab};
        \addlegendentry{$n = 15$}
        \addplot table [x=k, y=20] {tables/URB-stable-start.tab};
        \addlegendentry{$n = 20$}
        \addplot table [x=k, y=25] {tables/URB-stable-start.tab};
        \addlegendentry{$n = 25$}
        \addplot table [x=k, y=30] {tables/URB-stable-start.tab};
        \addlegendentry{$n = 30$}
      \end{axis}
    \end{tikzpicture}
  }
  \caption{Start of stable region.}
  \label{Fig:stable-start-k}
  \bigskip\bigskip
%
    \begin{tikzpicture}
      \begin{axis}[
          height=0.5\textwidth,
          width=0.95\textwidth,
          xlabel=$n$,
          ylabel=Mean distance from start,
          xmin=0,
          ymin=0,
          xmax=30,
          ymax=30,
          legend pos= north west
        ]
        \addplot table [x=n, y=URB-pos] {tables/stable-start.tab};
        \addlegendentry{uniformly random braids}
        \addplot table [x=n, y=word-pos] {tables/stable-start.tab};
        \addlegendentry{uniformly random words}
      \end{axis}
    \end{tikzpicture}
  \caption{Start of stable region (averages over all word lengths).}
  \label{Fig:stable-start-n}
\end{figure}

\medskip

Our observations thus lead us to make the following conjecture.

\begin{conjecture}[Stable region]\label{Conj:SR}
  Consider the braid monoid $B_n^+$ for any fixed $n \in \NN$.
  For $\mu_k = \Word_k$, respectively $\mu_k = \URB_k$,
  and for each $i$, the sequences of probability measures
  ${\lambda_i}_*(\mu_k)$ and ${\rho_i}_*(\mu_k)$ on the set of simple
  elements converge as $k \to \infty$.  Moreover, there exists a
  probability measure $\Sigma$ on the set of simple elements and
  constants $C$ and $D$ such that one has
  \begin{align*}
    \forall i > C &\quad {\lambda_i}_*(\mu_k) \to \Sigma \text{ as } k \to \infty \\
    \intertext{and}
    \forall i > D &\quad {\rho_i}_*(\mu_k) \to \Sigma \text{ as } k \to \infty
    \;.
  \end{align*}
\end{conjecture}

\subsection{Bounded expected penetration distance}\label{SS:NF:BoundedPD}%

We can view a uniformly random word as the result of a random process adding one
letter, chosen at random, at a time.  The stable region conjecture
suggests that the change to the normal form when multiplying by an
atom is unlikely to penetrate into the stable region.  This leads us
to investigate how the normal form of a random braid changes upon
multiplication by an atom.

The normal form of a word $wa$, where $w$ is in normal form and $a$ is an atom, can be
calculated from the normal form of $w$ by working through the word from the right to the left,
repeatedly applying the rewriting rule $xy \to (xm)(m^{-1}y)$
where $m = \partial x \wedge y$.  If at any point we have $m = \id$ then
we can stop.  If we have $xm = \Delta$ then all the following
rewrites will be of the form $x\Delta \to \Delta \tau(x)$, that is, consist of an application of the Garside automorphism; we will consider this a trivial change.

\begin{definition}
  For two braids $x$ and $y$ the penetration distance $\pd(x,y)$ for the product
  $xy$ is the number of simple factors at the end of the normal form
  of~$x$ which undergo a non-trivial change in the normal form of the
  product:
  \[
  \pd(x,y) = \cl(x) - \max\big\{ i \in \{0,\ldots,\cl(x)\} :
        x\Delta^{-\inf(x)}  \wedge \Delta^i
        = xy\Delta^{-\inf(xy)} \wedge \Delta^i \big\} \;.
  \]
\end{definition}

Using the same samples of uniformly random words and uniformly random braids as before, we
took each braid and calculated the penetration distances for its product with each
generator.  The mean penetration distance for each sample is
shown in \autoref{Fig:mean-pd}.  There are clear patterns here: the
mean penetration distance converges as the word length increases; the
value of the mean penetration distance increases with $n$; and it is
significantly larger for uniformly random braids than it is for uniformly random words.

\begin{figure}[!htbp]
  \centering
  \subfloat[Uniformly random words] {
    \centering
    \begin{tikzpicture}
      \begin{axis}[
          height=0.5\textwidth,
          width=0.96\textwidth,
          xlabel=Word length,
          ylabel=Mean $\pd$,
          xmin=0,
          ymin=0,
          xmax=4096,
        ]
        \addplot table [x=k, y=5] {tables/word-mean-pd.tab};
        \addlegendentry{$n = 5$}
        \addplot table [x=k, y=10] {tables/word-mean-pd.tab};
        \addlegendentry{$n = 10$}
        \addplot table [x=k, y=15] {tables/word-mean-pd.tab};
        \addlegendentry{$n = 15$}
        \addplot table [x=k, y=20] {tables/word-mean-pd.tab};
        \addlegendentry{$n = 20$}
        \addplot table [x=k, y=25] {tables/word-mean-pd.tab};
        \addlegendentry{$n = 25$}
        \addplot table [x=k, y=30] {tables/word-mean-pd.tab};
        \addlegendentry{$n = 30$}
      \end{axis}
    \end{tikzpicture}
  }

  \subfloat[Uniformly random braids]{
    \centering
    \begin{tikzpicture}
      \begin{axis}[
          height=0.5\textwidth,
          width=0.96\textwidth,
          xlabel=Word length,
          ylabel=Mean $\pd$,
          xmin=0,
          ymin=0,
          xmax=4096,
        ]
        \addplot table [x=k, y=5] {tables/URB-mean-pd.tab};
        \addlegendentry{$n = 5$}
        \addplot table [x=k, y=10] {tables/URB-mean-pd.tab};
        \addlegendentry{$n = 10$}
        \addplot table [x=k, y=15] {tables/URB-mean-pd.tab};
        \addlegendentry{$n = 15$}
        \addplot table [x=k, y=20] {tables/URB-mean-pd.tab};
        \addlegendentry{$n = 20$}
        \addplot table [x=k, y=25] {tables/URB-mean-pd.tab};
        \addlegendentry{$n = 25$}
        \addplot table [x=k, y=30] {tables/URB-mean-pd.tab};
        \addlegendentry{$n = 30$}
      \end{axis}
    \end{tikzpicture}
  }
  \caption{Mean penetration distance.}
  \label{Fig:mean-pd}

  \vskip4.13ex
  \subfloat[Uniformly random words]{
    \begin{tikzpicture}
      \begin{axis}[
          ycomb,
          no markers,
          height=0.45\textwidth,
          width=0.45\textwidth,
          xlabel=Generator,
          ylabel=Mean $\pd$,
          xmin=0,
          xmax=30,
          ymin=0,
        ]
        \addplot table [x=gen, y=Word] {tables/gen-mean-pd.tab};
      \end{axis}
    \end{tikzpicture}
  }
  \hfill
  \subfloat[Uniformly random braids]{
    \begin{tikzpicture}
      \begin{axis}[
          ycomb,
          no markers,
          height=0.45\textwidth,
          width=0.45\textwidth,
          xlabel=Generator,
          ylabel=Mean $\pd$,
          xmin=0,
          xmax=30,
          ymin=0,
        ]
        \addplot table [x=gen, y=URB] {tables/gen-mean-pd.tab};
      \end{axis}
    \end{tikzpicture}
  }
  \caption{Mean penetration distance for each generator for $n=30$ and word length $=2048$.}
  \label{Fig:gen-mean-pd}
\end{figure}

\begin{conjecture}[Uniformly bounded expected penetration distance]\label{Conj:BEPD}\hspace{0pt plus1em minus0pt}
  Consider the braid monoid $B_n^+$ for fixed $n \in \NN$, let $\mu_\Atoms$ be the uniform probability
  measure on the set of atoms and, for $k\in\NN$, let $\mu_k \in \{\Word_k, \URB_k\}$.
  Then there exists $C$ such that for all $k\in\NN$, we have
  \[ \mathbf{E}_{\mu_k \times \mu_\Atoms}[\pd] < C \;. \]
\end{conjecture}

\begin{corollary}
  There exists an algorithm to compute the normal form of a braid that
  has linear expected running time.
\end{corollary}

\begin{proof}
  Consider \autoref{NF:Alg} for computing the left normal form of a
  word.  The first loop is similar to the usual
  algorithm~\cite[Ch.\,9]{braid_epsteinetal} except that: we increment
  a counter~$I$ each time a~$\Delta$ is created, this will be the
  infimum of the normal form; we add a power of~$\Delta$ before each simple
  element, these are stored as exponents~$l_m$ modulo~$c$, where~$\Delta^c$ is the central power of~$\Delta$;
  and the inner loop stops as soon as a new~$\Delta$ is
  created.  We then add an additional pass, working backwards along the
  word, to push all the~$\Delta$s to the front.

  This algorithm has an invariant: the equality
  $x = \Delta^{I - \sum\limits_m \widetilde{l}_m} \Delta^{\widetilde{l}_1} s_1 \Delta^{\widetilde{l}_2} s_2
          \cdots \Delta^{\widetilde{l}_r} s_r$ in $B_n$,
  where the $\widetilde{l}_m\in\ZZ$ is a representative of $l_m$ for $m=1,2,\ldots,r$, remains true after each line has been completed.  Moreover, $l_m=0$ holds for $m=i,\ldots,r$ at any time.

  As the~$l_m$ are elements of $\ZZ/c\ZZ$, the operation of
  ``pushing'' the occurrences of~$\Delta$ over a simple element in lines \ref{NF:Alg:push1},
  \ref{NF:Alg:push2} and \ref{NF:Alg:push3} has bounded running time.

  The inner loop in line~\ref{NF:Alg:inner-loop} pushes the change triggered by multiplication by some~$x_l$ through the normal form of $x_1 x_2 \cdots x_{l-1}$, so the
  body of the loop will be executed $\pd(x_1 x_2 \cdots x_{l-1},
  x_l)$ times.  Hence, by \autoref{Conj:BEPD}, this loop is
  expected to take a constant amount of time.
  The two outer loops
  have at most~$k$ iterations with each iteration taking constant
  expected time, hence the whole algorithm has linear expected running
  time.
\end{proof}

\begin{algorithm}
  \caption{Calculate the normal form of a word $x = x_1 x_2 \cdots x_k$}
  \label{NF:Alg}
  \begin{algorithmic}[1]
    \REQUIRE $x_1, x_2, \ldots, x_k$ where each $x_i \in \Atoms$
    \ENSURE The left normal form $\Delta^I s_1 s_2 \cdots s_r$ of $x_1 x_2 \cdots x_k$
    \smallskip

    \STATE $l_1, l_2, \ldots, l_k \leftarrow 0 \in \ZZ/c\ZZ$; \quad
           $I \leftarrow 0 \in \ZZ$; \quad
           $s_j \leftarrow x_j$ for $j=1,2,\ldots,k$; \\
           $i \leftarrow 1$; \quad $r \leftarrow k$

    \WHILE{$i < r$}
      \IF{$\partial s_i \wedge s_{i+1} \ne \id$}
        \STATE $m \leftarrow \partial s_i \wedge s_{i+1}$; \quad
               $s_i \leftarrow s_i m$; \quad
               $s_{i+1} \leftarrow m^{-1} s_{i+1}$

        \STATE $j \leftarrow i$

        \IF{$j > 1$}
          \STATE $s_{j-1} \leftarrow \tau^{l_j}(s_{j-1})$; \quad
                 $l_{j-1} \leftarrow l_{j-1} + l_j$; \quad
                 $l_j \leftarrow 0$
                 \label{NF:Alg:push1}
        \ENDIF
        \WHILE{$s_j \ne \Delta$\label{NF:Alg:inner-loop}
               \AND $j > 1$
               \AND $\partial s_{j-1} \wedge s_j \ne \id$}
          \STATE $m \leftarrow \partial s_{j-1} \wedge s_j$; \quad
                 $s_{j-1} \leftarrow s_{j-1} m$; \quad
                 $s_j \leftarrow m^{-1} s_j$

          \STATE $j \leftarrow j-1$

          \IF{$j > 1$}
            \STATE $s_{j-1} \leftarrow \tau^{l_j}(s_{j-1})$; \quad
                   $l_{j-1} \leftarrow l_{j-1} + l_j$; \quad
                   $l_j \leftarrow 0$
                   \label{NF:Alg:push2}
          \ENDIF
        \ENDWHILE

        \IF{$s_j = \Delta$}
          \STATE $l_{j+1} \leftarrow l_j + l_{j+1} + 1$; \quad
                 $I \leftarrow I + 1$; \\
                 Delete $s_j$ and $l_j$, moving the following terms forward and decreasing $r$ by~1.
          \IF{$i > 1$}
            \STATE $i \leftarrow i - 1$
          \ENDIF
        \ENDIF

        \IF{$s_{i+1} = \id$}
          \STATE Delete $s_{i+1}$ and $l_{i+1}$, moving the following terms forward and decreasing $r$ by~1.
        \ELSE
          \STATE $i \leftarrow i+1$
        \ENDIF

      \ELSE
        \STATE $i \leftarrow i+1$
      \ENDIF
    \ENDWHILE

    \FOR{$j=r$ \TO $2$}
      \STATE $s_{j-1} \leftarrow \tau^{l_j}(s_{j-1})$; \quad
             $l_{j-1} \leftarrow l_{j-1} + l_j$; \quad
             $l_j \leftarrow 0$
             \label{NF:Alg:push3}
    \ENDFOR
  \end{algorithmic}
\end{algorithm}

\autoref{Fig:gen-mean-pd} shows the mean penetration distance of each
generator for $n=30$ and a word length of 2048.  A similar shape can
be seen for the other values of~$n$.  We see that not only is the mean
penetration distance longer for uniformly random braids, but also the
ratio of longest to shortest is significantly larger:  For uniformly random
words the ratio is less than~2, but for uniformly random braids it is greater
than~10.

\autoref{Fig:pd-dist} shows the distribution of penetration distances
observed in our sample for $n=30$ and a word length of 2048.
To understand the spike at $\pd=1$, consider a braid $x=s_1s_2\cdots s_k$ in normal form.
Multiplication of~$x$ on the right by a generator~$\sigma_i$ will modify the last factor $s_k$ if and only if $\sigma_i\notin F(s_k)$.
In contrast, in order for a modification of $s_j$ by a generator $\sigma_i\notin F(s_j)$ to trigger a modification of $s_{j-1}$, \emph{two} conditions must be met:
Firstly, one must have $S(s_j \sigma_i) \supsetneq S(s_j)$, that is, a new generator, say $\sigma_m$, must appear in the starting set $S(s_j \sigma_i)$.  (It can be shown that this happens if and only if the permutation~$\pi$ describing $s_k$ satisfies $\pi^{-1}(i+1) = \pi^{-1}(i)+1 = m+1$.)
Secondly, the new generator must be able to modify~$s_{j-1}$, that is, $\sigma_m\notin F(s_{j-1})$ must hold; this is the analogue of the condition for the case of the last factor~$s_k$.
The former condition, however, is absent in the case of the last factor~$s_k$, whence the counts for~$\pd = 1$ are higher, and those for $\pd=0$ lower, than what would be expected for a ``smooth'' distribution.

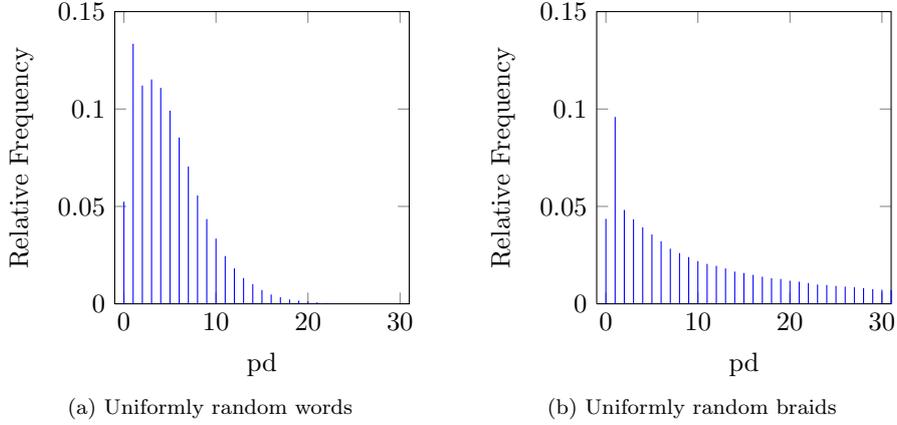
\begin{figure}[!htbp]
  \centering
  \subfloat[Uniformly random words]{
    \begin{tikzpicture}
      \begin{axis}[
          ycomb,
          no markers,
          height=0.45\textwidth,
          width=0.45\textwidth,
          xlabel=$\pd$,
          ylabel=Relative Frequency,
          xmin=-1,
          xmax=31,
          ymin=0,
          ymax=0.15,
          every y tick label/.append style  =
          { 
            /pgf/number format/.cd,
            precision = 2, 
            fixed
          }
        ]
        \addplot table [x=pd, y=freq] {tables/word-pd-dist.tab};
      \end{axis}
    \end{tikzpicture}
  }
  \hfill
  \subfloat[Uniformly random braids]{
    \begin{tikzpicture}
      \begin{axis}[
          ycomb,
          no markers,
          height=0.45\textwidth,
          width=0.45\textwidth,
          xlabel=$\pd$,
          ylabel=Relative Frequency,
          xmin=-1,
          xmax=31,
          ymin=0,
          ymax=0.15,
          every y tick label/.append style  =
          { 
            /pgf/number format/.cd,
            precision = 2, 
            fixed
          }
        ]
        \addplot table [x=pd, y=freq] {tables/URB-pd-dist.tab};
      \end{axis}
    \end{tikzpicture}
  }
  \caption{Distribution of penetration distances for $n=30$ and word length $=2048$.}
  \label{Fig:pd-dist}
\end{figure}

\section{Garside monoids}\label{S:GarsideGroups}%

Clearly the stable region conjecture and the bounded expected
penetration distance conjecture make sense in any Garside monoid and
for different sequences of probability measures.  We will give an
example of a Garside monoid $T_1$ where the
penetration distance is bounded, in other words there exists a
constant $C$ such that for any element $x$ and any atom $a$ we have
$\pd(x,a) < C$.  This stronger condition implies that both the bounded
penetration distance conjecture and the stable region conjecture hold
for $T_1$.  We will then go on to give a method to establish whether a variant
of the bounded penetration distance conjecture holds within a given
Garside monoid.

\subsection{A small Garside monoid}\label{SS:ToyGarsideGroup}%

\newcommand{\A}{A}
\newcommand{\B}{B}
\newcommand{\AB}{A\!B}
\newcommand{\BA}{B\!A}
\newcommand{\BB}{B\!B}
\newcommand{\BAB}{B\!A\!B}
\newcommand{\BBB}{B\!B\!B}

Let $T_1$ be the Garside monoid given by the presentation
\[ T_1 = \left\langle \A, \B \mid \A\,\B \A = \B \B \right\rangle^+ \;. \]
Its quotient group is isomorphic to the braid group on three
strands, but this monoid is distinct from both the classical and the dual Garside monoids~\cite{Dehornoy02}.  The
Garside element is $\BBB$ and there are eight simple elements.
\autoref{Fig:HasseDiagram} shows the structure of the prefix partial order on
the simple elements. \autoref{Fig:Transition} shows the matrix where
the entry~$(x,y)$ is~$1$ if $\partial x \wedge y = \id$ and~$0$
otherwise.  From this matrix one can easily read off which pairs of
simple elements are in left normal form.

\begin{figure}[!htbp]
  \centering
  \hfill
  \begin{minipage}{0.35\textwidth}%
    \centering
    \parbox[c]{\textwidth}{
      \centering
      \begin{xy}
        0;<4em,0em>:<0em,4em>::
        (1,0)*+{\id}="e";
        (0,1)*+{\A}="a";
        (2,1)*+{\B}="b";
        (0,2)*+{\AB}="ab";
        (2,2)*+{\BA}="ba";
        (0,3)*+{\BB}="bb";
        (2,3)*+{\BAB}="bab";
        (1,4)*+{\BBB}="bbb";
        {\ar@{->}^{\A} "e";"a"};
        {\ar@{->}_{\B} "e";"b"};
        {\ar@{->}^{\B} "a";"ab"};
        {\ar@{->}_{\B} "b";"bb"};
        {\ar@{->}_{\A} "b";"ba"};
        {\ar@{->}^{\A} "ab";"bb"};
        {\ar@{->}_{\B} "ba";"bab"};
        {\ar@{->}^{\B} "bb";"bbb"};
        {\ar@{->}_{\A} "bab";"bbb"};
      \end{xy}
    }
    \caption{Hasse diagram for the prefix partial order on the simple elements.} \label{Fig:HasseDiagram}
  \end{minipage}\hfill
  \begin{minipage}{0.55\textwidth}%
    \centering
    \parbox[c]{\textwidth}{
      \centering
      \begin{tabular}{C|CCCCCC}
             & \A & \B & \AB & \BA & \BB & \BAB \\ \hline
        \A   & 1  & 0  & 1   & 0   & 0   & 0    \\
        \B   & 0  & 0  & 0   & 0   & 0   & 0    \\
        \AB  & 0  & 1  & 0   & 1   & 0   & 1    \\
        \BA  & 1  & 0  & 1   & 0   & 0   & 0    \\
        \BB  & 1  & 0  & 1   & 0   & 0   & 0    \\
        \BAB & 0  & 1  & 0   & 1   & 0   & 1        
      \end{tabular}
    }
    \caption{The entry $(a,b)$ is 1 if $\partial a\wedge b=\id$ and 0 otherwise.} \label{Fig:Transition}
  \end{minipage}
  \hfill
\end{figure}

\begin{proposition} \label{Prop:Bounded}
  For any element $x \in T_1$ and any atom $a \in \{\A, \B\}$ we have
  \[ \pd(x,a) < 3 \;. \]
\end{proposition}

\begin{proof}
  First consider the atom $\A$ and look at how the change upon
  multiplication by $\A$ can penetrate through possible suffixes of
  the left normal form of~$x$.

  If we look at $s\A$ for each simple element suffix $s$ we see that
  there are three suffixes with canonical length one that have
  non-zero penetration distance:
  \[
  \begin{tabular}{l|CCCCCC}
    $s$ &
    \A  &
    \B  & 
    \AB &
    \BA &
    \BB &
    \BAB
    \\ \hline
    Normal form of $s\A$\rule[2.5ex]{0pt}{0pt} & 
    \A\,\A &
    \BA &
    \BB & 
    \BA\,\A &
    \BB\,\A &
    \Delta
    \\ \hline
    $\pd(s,\A)$\rule[2.5ex]{0pt}{0pt} &
    0 &
    1 &
    1 &
    0 &
    0 &
    1
  \end{tabular}
  \]
  So for a suffix $s$ with canonical length two to have a penetration
  distance
  greater than one the last factor cannot be $\A$, $\BA$, or $\BB$.
  We can also rule out $\BAB$ as in this case a $\Delta$ is created,
  which cannot affect any earlier factors.
  \autoref{Fig:Transition} indicates which pairs of simple elements can be adjacent in a left normal form;
  using this information, we can easily produce a list of possible suffixes:
  \[
  \begin{tabular}{l|CCCCC}
    $s$        &
    \AB\,\B    &
    \BAB\,\B   &
    \A\,\AB    &
    \BA\,\AB   &
    \BB\,\AB
    \\ \hline
    N.F. of $s \A$\rule[2.5ex]{0pt}{0pt} &
    \AB\,\BA     &
    \BAB\,\BA    &
    \AB\,\B      &
    \BAB\,\B     &
    \Delta\,\B
    \\ \hline
    $\pd(s,\A)$\rule[2.5ex]{0pt}{0pt} &
    1        &
    1        &
    2        &
    2        &
    2        
  \end{tabular}
  \]
  So for suffixes with canonical length three there are only two
  possibilities for the last two factors where the penetration distance could be
  greater than two: $\A\,\AB$ and $\BA\,\AB$.  This gives the
  following list of possible suffixes:
  \[
  \begin{tabular}{l|CCCCC}
    $s$            &
    \A\,\A\,\AB    &
    \BA\,\A\,\AB   &
    \BB\,\A\,\AB   &
    \AB\,\BA\,\AB  &
    \BAB\,\BA\,\AB 
    \\ \hline
    N.F. of $s \A$\rule[2.5ex]{0pt}{0pt} &
    \A\,\AB\,\B    &
    \BA\,\AB\,\B   &
    \BB\,\AB\,\B   &
    \AB\,\BAB\,\B  &
    \BAB\,\BAB\,\B 
    \\ \hline
    $\pd(s,\A)$\rule[2.5ex]{0pt}{0pt} &
    2        &
    2        &
    2        &
    2        &
    2        
  \end{tabular}
  \]
  Hence $\pd(x,\A)$ cannot be greater than two.

  We will now follow the same procedure for the atom $\B$.

  If we look at each simple element we see that there are four
  possible suffixes with canonical length one that have non-zero
  penetration distance:
  \[
  \begin{tabular}{l|CCCCCC}
    $s$ &
    \A  &
    \B  & 
    \AB &
    \BA &
    \BB &
    \BAB
    \\ \hline
    N.F. of $s B$\rule[2.5ex]{0pt}{0pt} & 
    \AB     &
    \BB     & 
    \AB\,\B &
    \BAB    &
    \Delta  &
    \BAB\,\B
    \\ \hline
    $\pd(s,A)$\rule[2.5ex]{0pt}{0pt} &
     1 &
     1 &
     0 &
     1 &
     1 &
     0
  \end{tabular}
  \]
  So for a suffix with canonical length two to have a penetration
  distance greater
  than one the last factor cannot be $\AB$, or $\BAB$.  We can also
  rule out $\BB$ as a $\Delta$ is produced.  This gives the following
  list of possible suffixes:
  \[
  \begin{tabular}{l|@{\;}C@{\quad}C@{\quad}C@{\quad}C@{\quad}C@{\quad}C@{\quad}C@{\;}}
    $s$        &
    \A\,\A     &
    \BA\,\A    &
    \BB\,\A    &
    \AB\,\B    &
    \BAB\,\B   &
    \AB\,\BA   &
    \BAB\,\BA  
    \\ \hline
    N.F. of $s \B$\rule[2.5ex]{0pt}{0pt} &
    \A\,\AB     &
    \BA\,\AB    &
    \BB\,\AB    &
    \Delta\,\A  &
    \Delta\,\BA &
    \AB\,\BAB   &
    \BAB\,\BAB  
    \\ \hline
    $\pd(s,\A)$\rule[2.5ex]{0pt}{0pt} &
    1         &
    1         &
    1         &
    2         &
    2         &
    1         &
    1         
  \end{tabular}
  \]
  All words with a penetration distance of two create a $\Delta$ so any preceding
  factors would not be affected.  Hence $\pd(x,\B)$ cannot be greater
  than two.
\end{proof} 

\paragraph{Remark} Since the Garside element $\Delta$ is central, any 
changes to the normal form when multiplying by an atom are limited to
a fixed number of simple factors at the end of the word.  This means
that the normal form can be computed in a single pass with a finite
state transducer (an automaton with output) which is augmented with an
integer counter to count each occurrence of $\Delta$.
  
\subsection{Penetration sequences}%

In the proof of \autoref{Prop:Bounded} we considered sequences of
simple elements in normal form and how the change penetrated through
them when we multiplied by an atom.  We can formalise this idea as
follows.
For the rest of this section, let $G^+$ be a Garside monoid with Garside element $\Delta$.

\begin{definition}
  A word $(s_k, m_k) \cdots (s_2, m_2) (s_1, m_1) \in
  \left(\pSimples \times \pSimples\right)^*$ is a \emph{penetration
    sequence} if, for all $i$, the following hold:
  \begin{gather}
    m_1 \preccurlyeq \partial s_1                                 \label{PS:absorbed} \\
    i < k \implies s_i m_i \ne \Delta                             \label{PS:no-Deltas} \\
    i < k \implies \partial s_{i+1} \wedge s_i = \id              \label{PS:s-normal}  \\
    i < k \implies m_{i+1} = \partial s_{i+1} \wedge s_{i} m_{i}  \label{PS:m} \;.
  \end{gather}

  Let $\PSeq_k$ denote the set of all penetration sequences of length $k$.
\end{definition}

For a penetration sequence, condition \eqref{PS:s-normal} ensures that
$s_k \cdots s_2 s_1$ is in normal form.  If we consider how this
normal form changes once we multiply $s_k \cdots s_2 s_1$ by $m_1$, then conditions \eqref{PS:absorbed} and  \eqref{PS:m} mean that $m_1$ moves into $s_1$ and, for $i<k$, $m_{i+1}$ is the simple factor that moves out of $s_{i}$ into $s_{i+1}$; in particular, one has $m_i \preccurlyeq \partial s_i$ for $i=1,2,\ldots,k$.

We are only interested in the canonical factors of the word and not
the initial power of $\Delta$.  Also, we are only interested in the
region where there is non-trivial movement between the factors.  So we
restrict to proper simple elements.

We only allow a $\Delta$ to be created at the very start of the
sequence, as otherwise the initial terms just correspond to conjugating
by $\Delta$ which we consider a trivial change; hence condition
\eqref{PS:no-Deltas}.
\medskip

For an element $x \in G^+$ and a simple element $s \in \Simples$ we
will say that a penetration sequence
$(s_k, m_k) \cdots (s_2, m_2) (s_1, m_1)$ is a \emph{penetration sequence for $xs$} if
$s_k \cdots s_2 s_1$ is a suffix of the normal form of $x$, and
$m_1 = \partial s_1 \wedge s$.  The penetration distance for the product $xs$ equals the length of the
longest penetration sequence for $xs$.

Let \[ G^+_{(k)} := \big\{ x \in G^+ : \cl(x) = k,\, \inf(x) = 0 \big\} \;. \]

\begin{proposition} \label{NF:prop_growth_rate}
  There exist constants $\alpha$, $\beta$, $p$, $q \ge 0$ such that
  \begin{align*}
    |\PSeq_k|  \in \Theta(k^p \alpha^k) \qquad\text{and}\qquad
    |G^+_{(k)}| \in \Theta(k^q \beta^k).
  \end{align*}
  Moreover, one has
  \begin{align*}
    \alpha = \lim_{k\to\infty} |\PSeq_k|^{1/k} \qquad\text{and}\qquad
    \beta = \lim_{k\to\infty} |G^+_{(k)}|^{1/k} \;,
  \end{align*}
  that is, the constants $\alpha$ and $\beta$ are the exponential growth rates
  of $|\PSeq_k|$ and $|G^+_{(k)}|$, respectively.
\end{proposition}
\begin{proof}
  Since a word in $\left(\pSimples \times \pSimples\right)^*$ is
  a penetration sequence if each consecutive pair of letters satisfy
  conditions \eqref{PS:absorbed}--\eqref{PS:m}, we have that $\PSeq_*
  = \bigcup_k \PSeq_k$ is a regular language.  Similarly, $G^+_{(*)} =
  \bigcup_k G^+_{(k)}$ is also a regular language.  Furthermore, these
  languages are factorial, that is any subword of a word in the
  language remains within the language.  The claim then follows by \cite[Corollary 4]{Shur08}.
\end{proof}

For the rest of this section, let $\alpha$, $\beta$, $p$ and $q$ as in \autoref{NF:prop_growth_rate}.

\begin{corollary} \label{NF:coro_growth_rate_comparison}
 One has either $\alpha < \beta$, or $\alpha=\beta$ and $p\le q$.
\end{corollary}

\begin{proof}
  The map $\iota: \PSeq_* \to G^+_{(*)} \times \pSimples$ defined by
  \[ 
    \iota \from (s_k, m_k)\cdots (s_2, m_2) (s_1, m_1) \mapsto (s_k \cdots s_2 s_1, m_1)
  \]
  gives an embedding of $\PSeq_k$ in $G^+_{(k)} \times \pSimples$ for every $k$.
  In particular, one has $|\PSeq_k| \in O(|G^+_{(k)}|)$ which implies the claim by
  \autoref{NF:prop_growth_rate}.
\end{proof}

\begin{corollary}\label{coro_PSeq_finite}
The penetration distance is bounded (that is, there exists a constant $C$ such that $\pd(x,s)<C$ holds for all
$x\in G^+$ and all $s\in\Simples$) if and only if $\alpha=0$.
\end{corollary}
\begin{proof}
The penetration distance is bounded if and only if $|\PSeq_k|=0$ for sufficiently large $k$, which is equivalent to $\alpha=0$ by \autoref{NF:prop_growth_rate}.
\end{proof}

\begin{example}\label{ex_Artin_type_I}\upshape
Consider the Artin group of type $\Artin{I}_2(p)$, denoting the classical Garside monoid \cite{braid_artin_groups} by $\Artin{I}_2(p)$ and the dual Garside monoid \cite{braid_dual} by~$\Artin{i}_2(p)$.
\autoref{Fig:typeI} shows the lattices of the simple elements with respect to $\preccurlyeq$.

\begin{figure}[!htbp]
  \subfloat[$\Artin{I}_2(p)$]{
      \begin{xy}
        0;<3em,0em>:<0em,3em>::
        (1,0)*+{\id}="e";
        (0,1)*+{a}="a";
        (2,1)*+{b}="b";
        (0,2)*+{ab}="ab";
        (2,2)*+{ba}="ba";
        (0,3)*+{aba}="aba";
        (2,3)*+{bab}="bab";
        (0,4)*+{\underbrace{aba\cdots}_{p-1 \text{ factors}}}="abab";
        (2,4)*+{\underbrace{bab\cdots}_{p-1 \text{ factors}}}="baba";
        (1,5)*+{\Delta}="D";
        {\ar@{->}^{a} "e";"a"};
        {\ar@{->}_{b} "e";"b"};
        {\ar@{->}^{b} "a";"ab"};
        {\ar@{->}_{a} "b";"ba"};
        {\ar@{->}^{a} "ab";"aba"};
        {\ar@{->}_{b} "ba";"bab"};
        {\ar@{..} "aba";"abab"};
        {\ar@{..} "bab";"baba"};
        {\ar@{->} "abab";"D"};
        {\ar@{->} "baba";"D"};
      \end{xy}
    }
  \hfill
  \subfloat[$\Artin{i}_2(p)$]{
      \raisebox{8mm}{\begin{xy}
        ;<3em,0em>:<0em,5em>::
        (3,0)*+{\id}="e";
        (0,1)*+{c_1}="a";
        (1,1)*+{c_2}="b";
        (3,1)*+{\cdots};
        (5,1)*+{c_{p-1}}="c";
        (6,1)*+{c_p}="d";
        (3,2)*+{\Delta}="D";
        {\ar@{->}^{c_1} "e";"a"};
        {\ar@{->}_{c_2} "e";"b"};
        {\ar@{->}^{c_{p-1}} "e";"c"};
        {\ar@{->}_{c_p} "e";"d"};
        {\ar@{->}^{c_2} "a";"D"};
        {\ar@{->}_{c_3} "b";"D"};
        {\ar@{->}^{c_p} "c";"D"};
        {\ar@{->}_{c_1} "d";"D"};
      \end{xy}}
    }
 \caption{Hasse diagrams for the Garside monoids $\Artin{I}_2(p)$ and $\Artin{i}_2(p)$.}
 \label{Fig:typeI}
\end{figure}

First consider $\Artin{I}_2(p)$.
Note that the starting set of an element $s\in\pSimples$ consists of a single atom.
Let $s_1,s_2\in \pSimples$, where $s_2$ is the product of~$\ell$ atoms.  It is obvious from the Hasse diagram that $\partial s_2\wedge s_1=\id$ holds if and only if either $\ell$ is even and $S(s_1)=S(s_2)$, or $\ell$ is odd and $S(s_1) = \Atoms\setminus S(s_2)$.
Hence, one has $|G^+_{(0)}| = 1$ and $|G^+_{(k)}| = 2(p-1)^k$ for $k>0$.
Moreover, one has
$\big|\{m_1 \in \pSimples : s_1m_1\preccurlyeq\Delta\}\big| = p-\ell$.
Summing over all possibilities for $s_1$ yields
$|\PSeq_1| = 2\sum_{\ell=1}^{p-1}(p-\ell) = p(p-1)$.  $|\PSeq_0| = 1$ holds trivially.
Finally, if $\partial s_2\wedge s_1=\id$ and $\Delta\not\preccurlyeq s_1m_1$ hold, then one has
$\partial s_2\wedge (s_1m_1)=\id$, whence one has $|\PSeq_k| = 0$ for $k>1$.
We thus obtain $\alpha=0$, $\beta=p-1$,
\[
 \sum_{k=0}^\infty |\PSeq_k| z^k = p(p-1)z + 1
 \qquad\text{and}\qquad
 \sum_{k=0}^\infty |G^+_{(k)}| z^k = \frac{-(p-1)z - 1}{(p-1)z - 1} \;.
\]

Now consider $\Artin{i}_2(p)$.  We have $\pSimples = \{c_1,c_2,\ldots,c_p\}$.
Moreover, $c_i c_j$ is simple if and only if $j=i+1$ modulo $p$, and $\partial c_i \wedge c_j=\id$ otherwise.
This immediately yields $|G^+_{(0)}| = 1$, $|G^+_{(1)}| = p$ and $|G^+_{(k)}| = p(p-1)^{k-1}$ for $k>1$, as well as $|\PSeq_0| = 1$, $|\PSeq_1| = p$ and $|\PSeq_k| = 0$ for $k>1$.
We thus obtain $\alpha=0$, $\beta=p-1$,
\[
 \sum_{k=0}^\infty |\PSeq_k| z^k = pz + 1
 \qquad\text{and}\qquad
 \sum_{k=0}^\infty |G^+_{(k)}| z^k = \frac{-z - 1}{(p-1)z - 1} \;.
\]

In particular, the Garside monoids $\Artin{I}_2(p)$ and $\Artin{i}_2(p)$ have bounded penetration distance by \autoref{coro_PSeq_finite}.
\end{example}

Even if there is no strict bound on the penetration distance $\pd(x,s)$ for a given Garside monoid, it is natural to ask whether the \emph{expected value} of the penetration distance is uniformly bounded in the sense of \autoref{Conj:BEPD} for a particular distribution of $x$ and $s$.  The following theorem gives a sufficient condition for this to be the case for natural distributions of $x$ and $s$.

\begin{theorem} \label{Thm:BoundedPenetrationDistance}
  Let $\nu_{k}$ be the uniform probability measure on $G^+_{(k)}$.
  If $\alpha < \beta$ then the expected value $\mathbf{E}_{\nu_{k} \times \mu_\Atoms}[\pd]$ of
  the penetration distance with respect to $\nu_k\times \mu_\Atoms$ is uniformly bounded (that is, the bound does not depend on $k$).
\end{theorem}

\begin{proof}
  If $\alpha < 1$ then, as $|\PSeq_k|$ only takes integer values, we
  have that $|\PSeq_k|$ is eventually $0$ and hence the penetration
  distance is bounded.  So we may assume that $\alpha \geq 1$.

  Let
  \[
    X_{i,k} := \big\{ (x,a) \in G^+_{(k)} \times \Atoms : \pd(x,a) = i \big\} \;,
  \]
  whence
  \[
    \mathbf{E}_{\nu_{k} \times \mu_\Atoms}[\pd] = 
    \sum_{i = 0}^k i 
      \frac{\left|X_{i,k}\right|}
           {|G^+_{(k)}|\cdot |\Atoms|} \;.
  \]

  Given $(x,a) \in X_{i,k}$ with $i>0$, let $(s_l, m_l) \cdots  (s_2, m_2) (s_1, m_1)$
  be a maximal penetration sequence for $xa$.  From the
  definition of $X_{i,k}$ we have that $l = i$ and, as $a$ is an atom,
  $m_1 = a$.  The simple factors $s_i,\ldots, s_2, s_1$ are the final
  $i$ factors of the normal form of $x$, so $x = x' s_i \cdots s_2 s_1$
  for some $x' \in G^+_{(k-i)}$.  This gives an injective map $X_{i,k}
  \hookrightarrow G^+_{(k-i)} \times \PSeq_i$, therefore we have
  \begin{equation} \label{EQ:BPD1}
    |X_{i,k}| \le |G^+_{(k-i)}|\cdot |\PSeq_i| \;.
  \end{equation}

  By \autoref{NF:prop_growth_rate}, there exist $C$ and $K$ such that for all
  $k > K$ we have
  \begin{align*}
    |\PSeq_k| \leq C k^p \alpha^k
    \qquad\text{and}\qquad
    \frac{1}{C} k^q \beta^k \leq |G^+_{(k)}| \leq C k^q \beta^k \;.
  \end{align*}

  For $k-i > K$ we thus have
  \begin{equation} \label{EQ:BPD2}
    \frac{|G^+_{(k-i)}|\cdot |\PSeq_i|}{|G^+_{(k)}|}
      \leq C^3 \frac{(k-i)^q \beta^{k-i}\ i^p \alpha^i}{k^q \beta^k}
      \leq C^3\, i^p \! \left(\frac{\alpha}{\beta}\right)^i \;.
  \end{equation}

  For $k-i \leq K$ and $k > 2K$ we have $i > K$ and thus
  \begin{align}
    \frac{|G^+_{(k-i)}|\cdot |\PSeq_i|}{|G^+_{(k)}|}
      \leq \frac{D |\PSeq_i|}{|G^+_{(k)}|}
      \leq C^2 D \frac{i^p \alpha^i}{k^q \beta^k}
      \leq C^2 D k^{p-q} \!\left(\frac{\alpha}{\beta}\right)^k \label{EQ:BPD3} \;,
  \end{align}
  where $D$ is the largest value of $|G^+_{(j)}|$ for $j \in \{1,2,\ldots,K\}$.

  By making use of \eqref{EQ:BPD1} and splitting the sum into two
  parts we obtain
  \[
    \mathbf{E}_{\nu_{k} \times \mu_\Atoms}[\pd] \leq
    \underbrace{
      \sum_{i = 1}^{k-K-1} i 
        \frac{|G^+_{(k-i)}|\cdot |\PSeq_i|}
             {|G^+_{(k)}|\cdot |\Atoms|}
    }_{S_k}
    +
    \underbrace{
      \sum_{i = k-K}^{k} i 
        \frac{|G^+_{(k-i)}|\cdot |\PSeq_i|}
             {|G^+_{(k)}|\cdot |\Atoms|}
    }_{T_k} \;.
  \]

  By \eqref{EQ:BPD2}, $k > K$ implies
  \[
    S_k 
      <
    \frac{C^3}{|\Atoms|}
    \sum_{i = 0}^{k-K-1}\!\! 
    i^{p+1} \! \left(\frac{\alpha}{\beta}\right)^i \;,
  \]
  and thus $S_k$ converges as $k \to \infty$ if $\alpha < \beta$.

  By \eqref{EQ:BPD3}, $k > 2K$ implies
  \[
    T_k
      <
    \frac{C^2 D}{|\Atoms|}
    \sum_{i = k-K}^{k} i 
    k^{p-q} \!\left(\frac{\alpha}{\beta}\right)^k
      <
    \frac{C^2 D(K+1)k^{p-q+1}}{|\Atoms|}
    \left(\frac{\alpha}{\beta}\right)^k \;,
  \]
  and thus $T_k$ converges as $k \to \infty$ if $\alpha < \beta$.

  Hence, if $\alpha < \beta$ then $\mathbf{E}_{\nu_{k} \times
    \mu_\Atoms}[\pd]$ is eventually bounded by a convergent sequence
  and so is bounded uniformly in $k$.
\end{proof}

The regular language $G^+_{(*)}\subset (\pSimples)^*$ is accepted by a deterministic finite automaton~$\Gamma$ defined as follows:
The set of states is $\mathcal{V}_\Gamma = \{ 1_\Gamma \} \cup \pSimples$, where $1_\Gamma$ is the initial state. The (partial) transition function $\mu_\Gamma:\mathcal{V}_\Gamma\times\pSimples\pto\pSimples\subset \mathcal{V}_\Gamma$
is given by
\begin{align*}
 \mu_\Gamma(1_\Gamma,s_1) &= s_1 \\[0.5ex]
 \mu_\Gamma(s_2,s_1) &=
   \begin{cases}
      s_1  & \text{if } \partial s_2 \wedge s_1 = \id \\
      \bot & \text{otherwise} 
   \end{cases}
\end{align*}
for all $s_1,s_2\in\pSimples$.
All states are accept states.

In particular, $|G^+_{(k)}|$ is the number of paths of length~$k$ in~$\Gamma$ that start at $1_\Gamma$.

Under an additional assumption on $\Gamma$, we can show that the condition from \autoref{Thm:BoundedPenetrationDistance} is necessary and sufficient for the expected penetration distance to be uniformly bounded.

\begin{theorem}\label{Thm:UnboundedPenetrationDistance}
Let $\nu_{k}$ be the uniform probability measure on $G^+_{(k)}$.  If $\Gamma\setminus\{1_\Gamma\}$ is strongly connected and $\alpha = \beta$ holds, then the expected value $\mathbf{E}_{\nu_{k} \times \mu_\Atoms}[\pd]$ of the penetration distance with respect to $\nu_k\times\mu_\Atoms$ satisfies
$\lim_{k\to\infty}\mathbf{E}_{\nu_{k} \times \mu_\Atoms}[\pd] = \infty$.
\end{theorem}

We prepare the proof of \autoref{Thm:UnboundedPenetrationDistance} with some auxiliary results.

For any integer $k$ and any atom $a\in\Atoms$ we define
\[
 \PSeq_k^a = \big\{ (s_k,m_k)\cdots(s_2,m_2)(s_1,m_1)\in\PSeq_k : m_1 = a \big\}
\]
and
\[
 \PSeq_k^\Atoms = \big\{ (s_k,m_k)\cdots(s_2,m_2)(s_1,m_1)\in\PSeq_k : m_1 \in \Atoms \big\}
    = \bigcup_{a\in\Atoms} \PSeq_k^a \;.
\]

\begin{lemma}\label{Lemma:PSeq_a}
One has $|\PSeq_k^\Atoms|\in\Theta(k^p \alpha^k)$, and $|\PSeq_k^a|\in O(k^p \alpha^k)$ holds for all $a\in\Atoms$.  Moreover, there exists $a\in\Atoms$ such that $|\PSeq_k^a|\in \Theta(k^p \alpha^k)$ holds.
\end{lemma}
\begin{proof}
For $w=(s_r,m_r)\cdots(s_1,m_1)\in\PSeq_*$ let $\rev(w)=(s_1,m_1)\cdots (s_r,m_r)$,
and for $S\subseteq \PSeq_*$ let $\rev(S)=\{ \rev(w) : w\in S \}$.
The set $\rev(\PSeq^\Atoms_*)$ is a regular language that is prefix-closed, whence there exist positive constants~$d$ and~$\gamma$, such that one has
$|\PSeq^\Atoms_k| = |\rev(\PSeq^\Atoms_k)| \in \Theta(k^{d} \gamma^k)$ by \cite[Corollary 4]{Shur08}.
As $\PSeq_k^\Atoms \subseteq \PSeq_k$ holds, we have $|\PSeq_k^\Atoms| \in O(|\PSeq_k|)$ and thus either $\gamma<\alpha$ or $\gamma=\alpha$ and $d\le p$ by \autoref{NF:prop_growth_rate}.

Similarly, for every $a\in\Atoms$, there exist positive constants $d_a$ and $\gamma_a$, such that
$|\PSeq^a_k| = |\rev(\PSeq^a_k)| \in \Theta(k^{d_a} \gamma_a^{\,k})$ holds, and $\PSeq_k^a \subseteq \PSeq_k^\Atoms$ implies that one has either $\gamma_a<\gamma$ or $\gamma_a=\gamma$ and $d_a\le d$.
Since $\PSeq_k^\Atoms = \bigcup_{a\in\Atoms} \PSeq_k^a$ holds and $\Atoms$ is finite, we have
$\gamma = \max_{a\in\Atoms}\{\gamma_a\}$ and $d = \max \{ d_a : a\in\Atoms, \gamma_a=\gamma\}$.

Hence, to complete the proof it is sufficient to prove $|\PSeq_k| \in O(|\PSeq_k^\Atoms|)$, for this would yield $\gamma=\alpha$ and $d=p$ with \autoref{NF:prop_growth_rate}.
To do this we will construct an injective map $\phi\from\PSeq_k \to \PSeq_k^\Atoms \times \Simples^2$ as follows.

For each $s \in \pSimples$ we fix a representative $w_s \in \Atoms^*$.
Now suppose we are given $S = (s_k,m_k) \cdots (s_2,m_2) (s_1,m_1) \in \PSeq_k$.
Let $x = s_k \cdots s_2 s_1$.
Using the representative $w_{m_1} = a_1 a_2 \cdots a_r$ we construct a sequence of elements $x_1,x_2,\ldots,x_{r+1}$ by defining $x_1=x$ and $x_{j+1} = x_j a_j$ for $j=1,2,\ldots,r$;
in particular, $x_{r+1}=xm_1$.

From $x_j \preccurlyeq x_{j+1}$ we obtain $\sup(x_j) \leq \sup(x_{j+1})$ for $j=1,2,\ldots,r$.
On the other hand, \eqref{PS:absorbed} yields $x m_1 \preccurlyeq s_k\cdots s_2\Delta$, and thus
$\sup(x_{r+1}) \le k=\sup(x_1)$.
Thus, $\sup(x_j)=k$ holds for $j=1,2,\ldots,r+1$, so there exist $s^{(j)}_1,s^{(j)}_2,\ldots,s^{(j)}_k\in\Simples\setminus\{\id\}$ such that $s_k^{(j)}\cdots s_2^{(j)} s_1^{(j)}$ is the normal form of $x_j$.

For $j=1,2,\ldots,r$, let
$ T_j = (s_{l_j}^{(j)}, m_{l_j}^{(j)}) \cdots (s_2^{(j)}, m_2^{(j)}) (s_1^{(j)}, m_1^{(j)}) \in \PSeq_{l_j}^\Atoms$
be a maximal penetration sequence for $x_j a_j$.
We have $l_j\le\sup(x_j)=k$, and $\sup(x_ja_j)=\sup(x_j)=k$ implies $a_j\preccurlyeq\partial s^{(j)}_1$, and thus $l_j>0$ and $m_1^{(j)} = a_j$.

As we have $s_k \neq s_km_k=s_k^{(r+1)}$, we can now fix $j\in\{1,2,\ldots,r\}$ minimal subject to the condition $s_k^{(j+1)} \ne s_k$.
Assume $l_j<k$ holds.
By minimality of~$j$, we have $s_k^{(j)}=s_k \neq s_k^{(j+1)}$, and thus
$\partial s_{l_j+1}^{(j)}\wedge s_{l_j}^{(j)} m_{l_j}^{(j)} \ne\id$.
Maximality of~$T_j$ implies $s_{l_j}^{(j)} m_{l_j}^{(j)} = \Delta$, and thus
\[
  \Delta \preccurlyeq s_{k-1}^{(j)} \cdots s_2^{(j)} s_1^{(j)} m_1^{(j)}
         = s_{k-1} \cdots s_2 s_1 a_1 a_2 \cdots a_j
         \preccurlyeq s_{k-1} \cdots s_2 s_1 m_1 \;.
\]
The latter implies $\Delta \preccurlyeq s_{k-1} m_{k-1}$, in contradiction to \eqref{PS:no-Deltas}.
Thus, we have $l_j=k$.

We can now define $\phi$ to be the map $S \mapsto (T_j, n_1, n_2)$ where $n_1 = a_1 a_2 \cdots a_{j-1}$
and $n_2 = a_{j+1} a_{j+2} \cdots a_r$.

As $l_j=k$ holds, the factors of $T_j$ determine $s^{(j)}_1,s^{(j)}_2,\ldots,s^{(j)}_k$, as well as $m^{(j)}_1$.
Since $m_1 = n_1 m_1^{(j)} n_2$ and
$s_k\cdots s_2 s_1 m_1 = s_k^{(j)}\cdots s_2^{(j)} s_1^{(j)} m_1^{(j)}n_2$ hold, we can recover~$S$ from~$T_j$, $n_1$ and $n_2$, proving that the map $\phi$ is injective.
\end{proof}

For $s\in\pSimples$ and $k\ge1$ we define $G^+_{(k)}(s) = G^+_{(k)} \cap \big(\pSimples\big)^{\!*\!}s = \{ s_k\cdots s_2 s_1 : s_1 = s \}$.  Recall the automaton~$\Gamma$ with set of states $\mathcal{V}_\Gamma = \{ 1_\Gamma \} \cup \pSimples$ accepting the regular language~$G^+_{(*)}$.
The elements of $G^+_{(k)}(s)$ correspond to the paths of length~$k-1$ in~$\Gamma$ that start at $s\in\mathcal{V}_\Gamma$.

\begin{lemma}\label{Lemma:GPlus_s}
If $\Gamma\setminus\{1_\Gamma\}$ is strongly connected, then $|G^+_{(k)}(s)| \in \Theta(k^q\beta^k)$ holds for every $s\in\pSimples$.
\end{lemma}
\begin{proof}
For $w=s_r\cdots s_2 s_1\in G^+_{(*)}$ define $\rev(w)=s_1 s_2\cdots s_r$,
and for $S\subseteq G^+_{(*)}$ let $\rev(S)=\{ \rev(w) : w\in S \}$.
For any $s\in\pSimples$, the set $\rev(G^+_{(*)}(s))$ is a regular language that is prefix-closed.
Thus, there exist positive constants~$d_s$ and~$\gamma_s$ such that one has
$|G^+_{(k)}(s)| = |\rev(G^+_{(k)}(s))| \in \Theta(k^{d_s} \gamma_s^{\,k})$ by \cite[Corollary~4]{Shur08}.

Let $s,t\in\pSimples$.  As $\Gamma\setminus\{1_\Gamma\}$ is strongly connected, there are $u_1,u_2,\ldots,u_r\in\pSimples$ satisfying $u_1=s$, $u_r=t$ and $u_1 u_2\cdots u_r\in G^+_{(*)}$, and thus $G^+_{(k)}(s)\,u_2\cdots u_r \subseteq G^+_{(k+r-1)}(t)$.
Hence, $k^{d_s} \gamma_s^k\in O((k+r-1)^{d_t} \gamma_t^{k+r-1}) = O(k^{d_t} \gamma_t^k)$ holds, whence one has either
$\gamma_s < \gamma_t$ or $\gamma_s = \gamma_t$ and $d_s \le d_s$.  As $s$ and $t$ were arbitrary, there are positive constants $d$ and $\gamma$ such that $\gamma_s=\gamma$ and $d_s=d$ hold for all $s\in\pSimples$.

The claim then follows from \autoref{NF:prop_growth_rate}, using
$G^+_{(k)} = \bigcup_{s\in\pSimples} G^+_{(k)}(s)$.
\end{proof}

\begin{proof}[Proof of \autoref{Thm:UnboundedPenetrationDistance}]
By \autoref{NF:prop_growth_rate}, \autoref{Lemma:PSeq_a} and \autoref{Lemma:GPlus_s}, and using the finiteness of $\pSimples$, there exist $a\in\Atoms$ and positive constants~$C$ and~$K$ such that for any $j\ge K$ one has $|\PSeq_j^a| \ge C j^p \alpha^j$, $|G^+_{(j)}| \le \frac1C j^q \beta^j$
and $|G^+_{(j)}(s)| \ge C j^q \beta^j$ for all $s\in\pSimples$.

For $k\ge 2K$ let $i=\lceil\frac{k}2\rceil$ and
$X_k = \{ s_k\cdots s_2 s_1 \in G^+_{(k)} : \pd(s_k\cdots s_2 s_1,a)\ge i \}$.  We define $\rho_k:X_k\to\PSeq_i^a$ by $\rho_k(s_k\cdots s_2 s_1) = (s_i,m_i)\cdots (s_2,m_2)(s_1,m_1)$, where $m_1=a\preccurlyeq \partial s_1$ and $m_{\ell+1} = \partial s_{\ell+1}\wedge s_\ell m_\ell \neq\id$ for $\ell=1,2,\ldots i-1$.
For any $S=(s_{i},m_{i})\cdots (s_2,m_2)(s_1,m_1)\in\PSeq_{i}^a$ and
$T=s_k\cdots s_{i}\in G^+_{(k-i+1)}(s_{i})$, one obtains $U=s_k\cdots s_2 s_1 \in X_k$ with $\rho_k(U)=S$.
As the map $(S,T)\mapsto U$ is obviously injective, one has
$|\rho_k^{-1}(S)| \ge |G^+_{(k-i+1)}(s_{i})| \ge C(k-i+1)^q\beta^{k-i+1}$,
and thus $|X_k| \ge |\PSeq_{i}^a|\cdot C(k-i+1)^q\beta^{k-i+1}
\ge C^2 i^p\alpha^{i} (k-i+1)^q \beta^{k-i+1}$.
We then obtain, using $\alpha=\beta$,
\begin{align*}
 \mathbf{E}_{\nu_{k} \times \mu_\Atoms}[\pd]
  &= \frac1{|G^+_{(k)}|\cdot|\Atoms|}\!\!\!\!\sum_{\mbox{\scriptsize$\begin{aligned}\quad x&\in G^+_{(k)}\\[-4pt]a'\!\!&\in\Atoms\end{aligned}$}}\!\!\!\!\!\pd(x,a')
  \ge \frac{i\cdot|X_k|}{|G^+_{(k)}|\cdot|\Atoms|} \\
  &\ge \frac{C^3\beta}{|\Atoms|}
          \cdot i^{p+1}\cdot\left(\frac{k-i+1}{k}\right)^q
  \ge \frac{C^3\beta}{2^{p+q+1}|\Atoms|} \cdot k^{p+1}
 \;,
\end{align*}
proving $\lim_{k\to\infty}\mathbf{E}_{\nu_{k} \times \mu_\Atoms}[\pd] = \infty$ as claimed.
\end{proof}

\subsection{Computing the exponential growth rates}%

Recall the automaton $\Gamma$ accepting the regular language $G^+_{(*)}$ from the previous section.
We have $G^+_{(k)} = \bigcup_{s\in\pSimples} G^+_{(k)}(s)$ and the elements of $G^+_{(k)}(s)$ correspond to the paths of length $k-1$ in $\Gamma$ that start at $s\in\pSimples$.
Hence, $|G^+_{(k)}|$ is the sum of all entries in the $(k-1)$-st power of the adjacency matrix $M_\Gamma$ of $\Gamma\setminus\{1_\Gamma\}$; the latter is the matrix whose entries, for $s_1,s_2\in\pSimples$, are given by
\[
(M_\Gamma)_{s_2,s_1} = \begin{cases}
  1 & \text{if } \partial s_2 \wedge s_1 = \id \\
  0 & \text{otherwise}
\end{cases}  \;.
\]

Similarly, the regular language $\PSeq_*\subset\mathcal{P}^*$, where $\mathcal{P}=\{ (s,m) \in \pSimples \times \pSimples : sm \preccurlyeq \Delta\}$ is accepted by a deterministic finite automaton $\Pi$ defined as follows:  
The set of states is
$\mathcal{V}_\Pi = \{1_\Pi\} \cup \mathcal{P}$, where $1_\Pi$ is the initial state.
The (partial) transition function
$\mu_\Pi:\mathcal{V}_\Pi \times \mathcal{P} \pto \mathcal{P} \subset \mathcal{V}_\Pi$ is given by
\begin{align*}
\mu_\Pi\big(1_\Pi,(s_1,m_1)\big) &= (s_1,m_1) \\[0.5ex]
\mu_\Pi\big((s_2,m_2),(s_1,m_1)\big) &=
  \begin{cases}
      (s_1,m_1) & \text{if } s_1 m_1 \ne \Delta, \,
                     \partial s_2 \wedge s_1 = \id, \\
                & \qquad\text{and } m_2 = \partial s_2 \wedge s_1 m_1 \quad\, \\
      \bot      & \text{otherwise}
  \end{cases}
\end{align*}
for all $(s_1,m_1),(s_2,m_2)\in\mathcal{P}$.  All states are accept states.

The number of strings of length $k$ in $\PSeq_*$, that is the number of paths of length~$k$ in~$\Pi$ that start at $1_\Pi$, is the sum of all entries in the $(k-1)$-st power of the adjacency matrix~$M_\Pi$ of $\Pi\setminus\{1_\Pi\}$; the latter is the matrix given by
\[
(M_\Pi)_{(s_2,m_2)(s_1,m_1)} =
   \begin{cases}
      1 & \text{if } s_1 m_1 \ne \Delta, \,
                 \partial s_2 \wedge s_1 = \id, \,
                 m_2 = \partial s_2 \wedge s_1 m_1 \\
      0 & \text{otherwise}
   \end{cases}
\]
for $(s_1,m_1),(s_2,m_2)\in\mathcal{P}$.

If $M\in\{M_\Pi,M_\Gamma\}$ and $\mathcal{L}$ is the language accepted by the respective automaton, then we have
$|\mathcal{L}_0| = 1$ and $|\mathcal{L}_k| = w M^{k-1} v$ for $k>0$,
where $v$ is the column vector with all entries~$1$ and $w$ is the row
vector with all entries~$1$.  Computing the minimal polynomial of $M$ and the initial terms of the
sequence $(|\mathcal{L}_k|)_k$ allows us, in principle, to compute the generating function of $\mathcal{L}$ exactly.  (The exponential growth rate of $\mathcal{L}$ is then the reciprocal of the radius
of convergence of the generating function.)
The results of applying this procedure to several different Garside
monoids are listed in \autoref{Tab:GrowthRates}.  Here $T_1 = \langle
\A, \B \mid \A\,\B \A = \A^2 \rangle^+$ is the Garside monoid from \autoref{SS:ToyGarsideGroup} and
$T_2 = \langle x,y,z \mid x z x y = y z x^2, y z x^2 z = z x y z x, z
x y z x = x z x y z \rangle^+$ is a Garside monoid described in
\cite{Picantin03}.  The remaining names denote Artin groups, with upper-case names referring to the corresponding classical Garside monoid \cite{braid_artin_groups} and lower-case names to the dual Garside monoid \cite{braid_dual}.

In practice, the above approach is limited by the size of the matrix $M_\Pi$.  Approximations of the exponential growth rates of $\PSeq_*$ and $G^+_{(*)}$, with guaranteed error bounds, can be obtained much more easily using an iterative algorithm~\cite[Algorithm 1]{Shur10}, which is based on the observation that if $P$ is an irreducible primitive matrix and $x$ is a vector with positive entries, then one has
\[
 \min_j\frac{(Px)_j}{x_j} \le \gamma \le \max_j\frac{(Px)_j}{x_j} \;,
\]
where $\gamma$ is the eigenvalue of $P$ of largest absolute value (which is a positive real number).
The growth rates in \autoref{Tab:GrowthRates} where no growth function is given were obtained using this method.

\begin{table}[!htbp]
\renewcommand{\arraystretch}{1.394}
  \centering\small
  \begin{tabular}{@{}CCCCC@{}} \toprule
    & \sum_{k=0}^\infty |\PSeq_k| z^k
    & \alpha
    & \sum_{k=0}^\infty |G^+_{(k)}| z^k
    & \beta \\ \midrule
    T_1
      & 8z^2 + 13z + 1
      & 0
      & \frac{-z^3 - 4z - 1}{2z - 1}
      & 2 \\
    T_2
      & \frac{-57z^3 - 102z^2 - 78z - 1}{z - 1}
      & 1
      & \frac{-11z^3 + 13z^2 - 9z - 1}{12z^3 - 16z^2 + 13z - 1}
      & 11.72\ldots \\
    \Artin{A}_2
      & 6z + 1
      & 0
      & \frac{-2z - 1}{2z - 1}
      & 2 \\ 
    \Artin{A}_3
      & \frac{4z^6 - \cdots
                + 97z + 1}{z^4 - 10z^3 + 15z^2 - 7z + 1}
      & 3.532\ldots
      & \frac{6z^3 - 3z^2 - 14z - 1}{6z^3 - 15z^2 + 8z - 1}
      & 5.449\ldots \\
    \Artin{A}_4
      & \frac{280z^{36} + \cdots 
                + 1592z + 1}%
             {4z^{34} + \cdots 
                 - 68z + 1}
      & 12.82\ldots
      & \frac{-144z^5 + \cdots
                 - 90z - 1}%
             {144z^5 - \cdots
                 + 28z - 1}
      & 18.71\ldots \\
    \Artin{A}_5
      &
      & 53.01\ldots
      &
      & 77.40\ldots \\
    \Artin{a}_2
      & 3z + 1
      & 0
      & \frac{-z - 1}{2z - 1}
      & 2 \\ 
    \Artin{a}_3
      & \frac{8z^5 - \cdots
                 + 23z + 1}%
             {4z^4 - 8z^3 + 8z^2 - 5z + 1}
      & 3.130\ldots
      & \frac{2z^4 + 4z^3 - 5z^2 + 4z + 1}{10z^4 - 20z^3 + 19z^2 - 8z + 1}
      & 4.839\ldots \\
    \Artin{a}_4
      & \frac{-2457600z^{31} + \cdots 
                 + 154z + 1}%
             {1638400z^{30} - \cdots 
                 - 36z + 1}
      & 8.822\ldots
      & \frac{-40z^8 - \cdots
                 + 16z + 1}%
             {560z^8 - \cdots
                 - 24z + 1}
      & 12.83\ldots \\
    \Artin{a}_5
      &
      & 25.31\ldots
      &
      & 35.98\ldots \\
    \Artin{a}_6
      &
      & 73.95\ldots
      &
      & 104.87\ldots \\
    \Artin{B}_2
      & 12z + 1 & 0
      & \frac{-3z - 1}{3z - 1}
      & 3 \\ 
    \Artin{B}_3
      & \frac{270z^{13} - \cdots
                 - 338z - 1}
             {9z^{11} - \cdots
                 + 24z - 1}
      & 6.44\ldots
      &  \frac{60z^4 - 77z^3 - 59z^2 + 27z + 1}{60z^4 - 149z^3 + 91z^2 - 19z + 1}
      & 12.75\ldots \\
    \Artin{B}_4
      &
      & 38.21\ldots
      &
      & 71.39\ldots \\
    \Artin{b}_2
      & 4z + 1
      & 0
      & \frac{-z - 1}{3z - 1}
      & 3 \\ 
    \Artin{b}_3
      & \frac{126z^7 - \cdots
                 + 34z + 1}
             {36z^6 - \cdots
                 - 11z + 1}
      & 5.006\ldots
      & \frac{6z^5 + \cdots
                 - 3z - 1}
             {60z^5 - \cdots
                 + 15z - 1}
      & 9.568\ldots \\
    \Artin{b}_4
      &
      & 18.02\ldots
      &
      & 31.29\ldots \\
    \Artin{b}_5
      &
      & 61.74\ldots
      &
      & 104.07\ldots \\
    \Artin{b}_6
      &
      & 207.73\ldots
      &
      & 350.24\ldots \\
    \Artin{D}_2
      & 2z + 1
      & 0
      & \frac{-z - 1}{z - 1}
      & 1 \\ 
    \Artin{D}_3
      & \frac{4z^6 - \cdots
                 + 97z + 1}
             {z^4 - 10z^3 + 15z^2 - 7z + 1}
      & 3.532\ldots
      & \frac{6z^3 - 3z^2 - 14z - 1}{6z^3 - 15z^2 + 8z - 1}
      & 5.449\ldots \\ 
    \Artin{D}_4
      & \frac{14592z^{24} - \cdots
                  + 3505z + 1}
             {16z^{22} - \cdots
                  - 71z + 1}
      & 19.66\ldots
      & \frac{-360z^5 + \cdots
                  - 146z - 1}
             {360z^5 - \cdots
                  + 44z - 1}
      & 32.68\ldots \\
    \Artin{d}_2
      & 2z + 1
      & 0
      & \frac{-z - 1}{z - 1}
      & 1 \\ 
    \Artin{d}_3
      & \frac{8z^5 - \cdots
                 + 23z + 1}
             {4z^4 - 8z^3 + 8z^2 - 5z + 1}
      & 3.130\ldots
      & \frac{2z^4 + 4z^3 - 5z^2 + 4z + 1}{10z^4 - 20z^3 + 19z^2 - 8z + 1}
      & 4.839\ldots \\ 
    \Artin{d}_4
      & \frac{1920z^{15} + \cdots
                 + 213z + 1}
             {6400z^{14} - \cdots
                 - 24z + 1}
      & 11.24\ldots
      & \frac{-20z^7 - \cdots
                 - 22z - 1}
             {400z^7 - \cdots
                 + 26z - 1}
      & 18.40\ldots \\
    \Artin{d}_5
      &
      & 40.05\ldots
      &
      & 66.34\ldots \\
    \Artin{d}_6
      &
      & 139.58\ldots
      &
      & 234.58\ldots \\
    \Artin{e}_6
      &
      & 181.24\ldots
      &
      & 339.75\ldots \\
    \Artin{f}_4
      & \frac{-561126113280z^{40} + \cdots
                 - 505z - 1}
             {96745881600z^{39} - \cdots
                 + 66z - 1}
      & 26.92\ldots
      & \frac{-720z^{10} - \cdots
                 + 28z + 1}
             {47520z^{10} - \cdots
                 - 75z + 1}
      & 60.55\cdots \\
    \Artin{G}_2
      & 30z + 1
      & 0
      & \frac{-5z - 1}{5z - 1}
      & 5 \\
    \Artin{g}_2
      & 6z + 1
      & 0
      & \frac{-z - 1}{5z - 1}
      & 5 \\
    \Artin{H}_3
      & \frac{1364z^{20} - \cdots
                 + 1758z + 1}
             {4z^{18} - \cdots
                 - 34z + 1}
      & 13.02\ldots
      & \frac{72z^4 - 196z^3 + 77z^2 + 76z + 1}{72z^4 - 244z^3 + 229z^2 - 42z + 1}
      & 35.79\ldots \\
    \Artin{h}_3
      & \frac{400z^8 - \cdots
                 - 67z - 1}
             {64z^7 - \cdots
                 + 13z - 1}
      & 7.924\ldots
      & \frac{8z^6 + \cdots
                 + 4z + 1}
             {168z^6 - \cdots
                 - 26z + 1}
      & 20.10\ldots \\
    \Artin{h}_4
      &
      & 61.56\ldots
      &
      & 217.25\ldots \\
    \Artin{I}_2(p)
      & p(p-1)z + 1
      & 0
      & \frac{-(p-1)z - 1}{(p-1)z - 1}
      & p-1 \\
    \Artin{i}_2(p)
      & pz + 1
      & 0
      & \frac{-z - 1}{(p-1)z - 1}
      & p-1
  \end{tabular}
  \caption{Generating functions and exponential growth rates.}
  \label{Tab:GrowthRates}
\end{table}

In all the examples listed in \autoref{Tab:GrowthRates} one has $\alpha < \beta$, that is, all of these monoids satisfy the conditions of \autoref{Thm:BoundedPenetrationDistance}.  We can also see that, in
addition to $T_1$, the Garside monoids corresponding to Coxeter groups with 2 simple reflections also have the property that the penetration distance (not just its expected value) is bounded.
We also remark that, with the exception of the monoids $T_1$, $T_2$, $\Artin{D}_2$ and $\Artin{d}_2$, all monoids listed in \autoref{Tab:GrowthRates} have the property that the acceptor of $G^+_{(*)}$ is strongly connected.  For all classical Garside monoids of type $\Artin{A}$ this follows from \cite[Lemma 3.4]{Caruso13}.

\autoref{Fig:alpha-beta} shows plots of $\alpha$ against $\beta$ for
the monoids listed in \autoref{Tab:GrowthRates}.  Lines are drawn trough the points for the two largest monoids in each series.  While the number of points is small, the data is very suggestive; we conjecture that, in each series, the points asymptotically lie on a straight line above the diagonal.  This would, in particular, imply that the conditions of \autoref{Thm:BoundedPenetrationDistance} are satisfied, and thus the expected penetration distance is uniformly bounded, for all monoids in the series considered.

\begin{figure}[!htbp]
\centering
\subfloat[Type $\Artin{A}$.]{
  \begin{tikzpicture}[spy using outlines={
          circle,
          magnification=3,
          connect spies
      }]
    \begin{axis}[
        height=0.7\textwidth,
        width=0.9\textwidth,
        xlabel=$\alpha$,
        ylabel=$\beta$,
        xmin=0,
        ymin=0,
      ]
      \addplot[
        scatter,mark=*,only marks, 
        point meta=\thisrow{color},
        nodes near coords*={\ifnum\label=1 $\expandafter\Artin\group$\fi},
        visualization depends on={value \thisrow{group} \as \group},
        visualization depends on={value \thisrow{label} \as \label}
      ] 
      table [x=alpha, y=beta] {tables/growth_rates_A.tab};
      \addplot [domain=0:75, samples=2] {1.46068*x - 0.02049};  
      \addplot [domain=0:75, samples=2] {1.4161*x + 0.1339};    

      \coordinate (a2) at (axis cs:0,2);
      \coordinate (a3) at (axis cs:3.53,5.45);
      \coordinate (a6) at (axis cs:73.96,104.87);

      \coordinate (spypoint) at (axis cs:2,3);
      \coordinate (magnifyglass) at (axis cs:63,38);
    \end{axis}
    \spy [gray, size=4cm] on (spypoint)
                            in node[fill=white] at (magnifyglass);

    \node[below right] at (a2) {$\Artin{a}_2$};
    \node[below right] at (a3) {$\Artin{a}_3$};
    \node[below right] at (a6) {$\Artin{a}_6$};
  \end{tikzpicture}
}

\subfloat[Types $\Artin{B}$ and $\Artin{E}$.]{
  \begin{tikzpicture}[spy using outlines={
          circle,
          magnification=3,
          connect spies
      }]
    \begin{axis}[
        height=0.7\textwidth,
        width=0.9\textwidth,
        xlabel=$\alpha$,
        ylabel=$\beta$,
        xmin=0,
        ymin=0,
      ]
      \addplot[
        scatter,mark=*,only marks, 
        point meta=\thisrow{color},
        nodes near coords*={\ifnum\label=1 $\expandafter\Artin\group$\fi},
        visualization depends on={value \thisrow{group} \as \group},
        visualization depends on={value \thisrow{label} \as \label}
      ] 
      table [x=alpha, y=beta] {tables/growth_rates_BE.tab};
      \addplot [domain=0:110, samples=2] {1.8461*x + 0.8612};    
      \addplot [domain=0:210, samples=2] {1.68621*x - 0.05355};  

      \coordinate (b2) at (axis cs:0,2);
      \coordinate (b3) at (axis cs:5.01,9.57);

      \coordinate (spypoint) at (axis cs:5,8);
      \coordinate (magnifyglass) at (axis cs:170,125);
    \end{axis}
    \spy [gray, size=4cm] on (spypoint)
                            in node[fill=white] at (magnifyglass);

    \node[below right] at (b2) {$\Artin{b}_2$};
    \node[right] at (b3) {$\Artin{b}_3$};
  \end{tikzpicture}
}
  \caption{Scatter plot of the exponential growth rates.}
  \label{Fig:alpha-beta}
\end{figure}

\begin{figure}[!htbp]
\ContinuedFloat
\centering

\subfloat[Type $\Artin{D}$.]{
  \begin{tikzpicture}[spy using outlines={
          circle,
          magnification=3,
          connect spies
      }]
    \begin{axis}[
        height=0.7\textwidth,
        width=0.9\textwidth,
        xlabel=$\alpha$,
        ylabel=$\beta$,
        xmin=0,
        ymin=0,
      ]
      \addplot[
        scatter,mark=*,only marks, 
        point meta=\thisrow{color},
        nodes near coords*={\ifnum\label=1 $\expandafter\Artin\group$\fi},
        visualization depends on={value \thisrow{group} \as \group},
        visualization depends on={value \thisrow{label} \as \label}
      ] 
      table [x=alpha, y=beta] {tables/growth_rates_D.tab};
      \addplot [domain=0:140, samples=2] {1.6888*x - 0.5114};    
      \addplot [domain=0:140, samples=2] {1.691*x - 1.386 };     

      \coordinate (d2) at (axis cs:0,1);
      \coordinate (d3) at (axis cs:3.13,4.84);

      \coordinate (spypoint) at (axis cs:2,3);
      \coordinate (magnifyglass) at (axis cs:115,80);
    \end{axis}
    \spy [gray, size=4cm] on (spypoint)
                            in node[fill=white] at (magnifyglass);

    \node[below right] at (d2) {$\Artin{d}_2$};
    \node[right] at (d3) {$\Artin{d}_3$};
  \end{tikzpicture}
}

\subfloat[Other monoids.]{
  \begin{tikzpicture}[spy using outlines={
          circle,
          magnification=3,
          connect spies
      }]
    \begin{axis}[
        height=0.7\textwidth,
        width=0.9\textwidth,
        xlabel=$\alpha$,
        ylabel=$\beta$,
        xmin=0,
        ymin=0,
      ]
      \addplot[
        scatter,mark=*,only marks, 
        point meta=\thisrow{color},
        nodes near coords*={\ifnum\label=1 $\expandafter\Artin\group$\fi},
        visualization depends on={value \thisrow{group} \as \group},
        visualization depends on={value \thisrow{label} \as \label}
      ] 
      table [x=alpha, y=beta] {tables/growth_rates_other.tab};
      \addplot [domain=0:65, samples=2] {3.676*x - 9.011};      

      \coordinate (T1) at (axis cs:0,2);
      \coordinate (T2) at (axis cs:1,11.72);

      \coordinate (spypoint) at (axis cs:1,7);
      \coordinate (magnifyglass) at (axis cs:52,77);
    \end{axis}
      \spy [gray, size=4cm] on (spypoint)
                            in node[fill=white] at (magnifyglass);

    \node[below right] at (T1) {$T_1$};
    \node[above right] at (T2) {$T_2$};
  \end{tikzpicture}
}
  \caption{Scatter plot of the exponential growth rates (ctd.).}
\end{figure}

\section*{Acknowledgements}%
The authors thank the technical staff of the School of Computing, Engineering and Mathematics at the University of Western Sydney for their outstanding support, which significantly contributed to this project.

\newcommand{\etalchar}[1]{$^{#1}$}

\bigskip

\begin{minipage}[t]{0.45\textwidth}
\noindent\textbf{Volker Gebhardt}\\
\noindent E-mail: \texttt{v.gebhardt@uws.edu.au}
\end{minipage}
\hfill
\begin{minipage}[t]{0.5\textwidth}
\noindent\textbf{Stephen Tawn}\\
\noindent E-mail: \texttt{stephen@tawn.co.uk}\\
\noindent URL: \url{http://www.stephentawn.info}
\end{minipage}
\medskip
\begin{center}
University of Western Sydney\\
Centre for Research in Mathematics\\
Locked Bag 1797, Penrith NSW 2751, Australia\\
\noindent URL: \url{http://www.uws.edu.au/crm}
\end{center}

\end{document}